\documentclass{amsart}
\usepackage{amsmath}
\usepackage{amssymb}
\usepackage{amsthm}
\usepackage{xcolor}
\usepackage[margin=1.2in]{geometry}
\usepackage{mathtools}
\usepackage{graphicx}
\usepackage{enumerate}
\usepackage[ruled]{algorithm2e}
\usepackage{soul}
\usepackage{subcaption}
\usepackage{hyperref}
\usepackage{bm}
\usepackage{epstopdf}
\graphicspath{{figures/}}

\DeclareMathOperator*{\dif}{d\!}

\newtheorem{thm}{Theorem}
\newtheorem{defn}[thm]{Definition}

\newtheorem{prop}[thm]{Proposition}

\newtheorem{remark}[thm]{Remark}

\author[Yang]{Tianyu Yang}
\address{Department of Computational Mathematics Science and Engineering, Michigan State University, East Lansing, MI 48824, USA}
\curraddr{}
\email{yangti27@msu.edu}

\author[Yang]{Yang Yang}
\address{Department of Computational Mathematics Science and Engineering, Michigan State University, East Lansing, MI 48824, USA}
\curraddr{}
\email{yangy5@msu.edu}

\thanks{The research of T. Yang and Y. Yang is partially supported by the NSF grants DMS-2006881, DMS-2237534, DMS-2220373, and the NIH grant R03-EB033521.
}
 
\title[UMBLT with Partial Data and Uncertain Optical Parameters]{The Diffusive Ultrasound Modulated Bioluminescence Tomography with Partial Data and Uncertain Optical Parameters}

\begin{document}

\begin{abstract}
The paper studies an imaging problem in the diffusive ultrasound-modulated bioluminescence tomography with partial boundary measurement in an anisotropic medium. Assuming plane-wave modulation, we transform the imaging problem to an inverse problem with internal data, and derive a reconstruction procedure to recover the bioluminescent source. Subsequently, an uncertainty quantification estimate is established to assess the robustness of the reconstruction. To facilitate practical implementation, we discretize the diffusive model using the staggered grid scheme, resulting in a discrete formulation of the UMBLT inverse problem. A discrete reconstruction procedure is then presented along with a discrete uncertainty quantification estimate.
Finally, the reconstruction procedure is quantitatively validated through numerical examples to demonstrate the efficacy and reliability of the proposed approach and estimates.
\end{abstract}

\maketitle

\section{Introduction and Problem Formulation}

Bioluminescence refers to production and emission of native light inside living organisms such as fireflies. Based on this phenomenon, Bio-Luminescence Tomography (BLT) is developed as a technology that utilizes bioluminescence sources as bio-medical indicators to image biological tissue. Specifically, biological entities or process components (e.g. bacteria, tumor cells, immune cells, or genes) are tagged in BLT with reporter genes that encode one of a number of light-generating enzymes (luciferases)~\cite{contag2002advances}.
By measuring the light generated by the luciferin-luciferase reaction, BLT aims to image the spatial distribution of the internal bioluminescence sources.

\medskip
\textbf{The Inverse Problem in Diffusive BLT.}
Let $\Omega$ represent the strongly-scattering biological tissue. We will assume $\Omega$ is a bounded connected open subset of $\mathbb{R}^n$ with smooth boundar $\partial\Omega$. The light propagates in a strongly-scattering medium as a diffuse wave~\cite{arridge2009optical}. The spatial photon density $\phi=\phi(x)$ of the wave is modeled by the following time-independent diffusion equation with the Robin-type boundary condition~\cite{bal2014ultrasound}:
\begin{align}
    -\nabla\cdot D(x) \nabla \phi(x) + \sigma_{a}(x) \phi(x) & = S(x)\quad &\text{ in } \Omega. \label{eq:DE} \\
    \phi(x) +\ell\nu\cdot D(x)\nabla\phi(x) & = 0 &\text{ on } \partial \Omega. \label{eq:DEbc}
\end{align}
Here, $D=D(x)$ is the \textit{diffusion coefficient}, $\sigma_a=\sigma_a(x)$ is the \textit{absorption coefficient}, $S=S(x)$ is the spatial distribution of the \textit{bio-luminescence source}, $\ell$ is the \textit{extrapolation length}, and $\nu$ is the unit outer normal vector field to $\partial\Omega$. 
Henceforth, we will assume that the light intensity is measured only over a narrow band of frequencies, so that the diffusion coefficient $D$ and the absorption coefficient $\sigma_a$ are frequency-independent. The inverse problem in BLT can be stated as follows: given $D(x)$ and $\sigma_a(x)$, recover the internal source $S(x)$ from the boundary photon density $\phi|_{\Gamma}$ measured on an open subset of the boundary $\Gamma\subset\partial\Omega$.

\medskip
\textbf{Ultrasound Modulation.} The measurement in diffusive BLT alone is insufficient to uniquely identify the bio-luminescence source. This is clear from the above formulation, as the inverse problem in BLT is a classical inverse source problem that is well known to lack unique solutions~\cite{isakov1990inverse}. Diffusive BLT typically suffers from limited spatial resolution due to strong scattering of light in soft tissue. Various methods have been proposed to enhance the identifiability and spatial resolution of the bio-luminescence source. One of them~\cite{huynh2013ultrasound} makes use of a focused ultrasound beam to modulate BLT and generate additional data. 
Here, ultrasound modulation means performing the usual BLT measurement while the medium
undergoes a series of acoustic perturbation.

In the literature, two distinct models have been proposed for ultrasound modulation. One involves modulation with spherical waves, as detailed in~\cite{ammari2012acousto}, where the displacement function from a short diverging spherical acoustic impulse is derived. This model finds application in the analysis of ultrasound modulation across electromagnetic tomography~\cite{ammari2012acousto}, diffuse optical tomography~\cite{AmmBosGarNguSep}, and acousto-optic imaging~\cite{AmmNguSep}. 
The other model involves modulation with plane waves, for which the displacement function is calculated in~\cite{bal2010inverse}. This model has been studied, for instance, in the analysis of ultrasound modulated bio-luminescence tomography~\cite{bal2016ultrasound, bal2014ultrasound, chung2020ultrasound}, optical tomography~\cite{bal2014local, ChuHosSchRTE, ChuSch,chung2020coherent, chung2017inverse, li2019hybrid, li2020inverse}, and acousto-electromagnetic imaging~\cite{bal2013cauchy,li2021acousto,li2023inverse}.
In this paper, we will assume plane-wave modulation.

Suppose the incident plane wave is of the form $\cos(q\cdot x + \varphi)$ where $q\in\mathbb{R}^n$ is the wave vector and $\varphi$ is the phase. The time scale of the acoustic field propagation is generally much greater than that of the optical field, hence the acoustic field can effectively modulate the optical field.
Following~\cite{bal2014ultrasound}, the effect of the acoustic modulation on the aforementioned optical parameters takes the form:
\begin{align}
    D_\varepsilon(x)&\coloneqq (1 + \varepsilon(2\gamma-1) \cos(q\cdot x + \varphi))D(x), \label{eq:mD}\\
    \sigma_{a,\varepsilon}(x)&\coloneqq (1 + \varepsilon(2\gamma+1) \cos(q\cdot x + \varphi))\sigma_a(x), \label{eq:msigma} \\
    S_\varepsilon(x)&\coloneqq (1+\varepsilon \cos(q\cdot x + \varphi))S(x), \label{eq:mS}
\end{align}
where $\gamma$ is the elasto-optical constant, 
$0\leq\varepsilon\ll1$ is a small parameter related to the amplitude, frequency, time, density and acoustic wave speed~\cite{bal2014ultrasound}.

\medskip
\textbf{Inverse Problem in Diffusive Ultrasound Modulated BLT (UMBLT).} In the presence of ultrasound modulation, the optical parameters and the bio-luminescence source are modulated according to~\eqref{eq:mD}-\eqref{eq:mS}. The diffusion equation for the modulated photon density $\phi_\varepsilon$ reads
\begin{alignat}{2}
    -\nabla\cdot D_\varepsilon(x) \nabla \phi_\varepsilon(x) + \sigma_{a,\varepsilon}(x) \phi_\varepsilon(x) & = S_\varepsilon(x)\quad &&\text{ in }  \Omega. \label{eq:mDE} \\
    \phi_\varepsilon +\ell\nu\cdot D_\varepsilon\nabla\phi_\varepsilon & = 0 &&\text{ on } \partial \Omega. \label{eq:mDEbc}
\end{alignat}
We will write $D_0, \sigma_{a,0}, \phi_0$ for the quantities without modulation, that is, when $\varepsilon=0$. 
The measurement in UMBLT is the modulated boundary photon density on an open subset of the boundary $\Gamma\subset\partial\Omega$:
\begin{equation} \label{eq:DEmeas}
\Lambda_{\varepsilon, q, \varphi}[S] := \phi_\varepsilon|_{\Gamma}, \quad\quad \text{ for any } q\in\mathbb{R}^n, \; \varepsilon\geq 0.
\end{equation}
We refer to the measurement as \textit{full data} if $\Gamma=\partial\Omega$ and \textit{partial data} if $\Gamma\subsetneq\partial\Omega$. Note that assuming such measurement, the modulated boundary photon current $\nu\cdot D_\varepsilon \nabla\phi_{\varepsilon}|_{\Gamma}$ is readily known on $\Gamma$ in view of the relation~\eqref{eq:mDEbc}.
Therefore, the inverse problem in UMBLT is to recover the bio-luminescence source $S$ from the measurement~\eqref{eq:DEmeas}, assuming $D$ and $\sigma_a$ are given.

\medskip
\textbf{Literature Review.} 
We briefly review the literature on mathematical inverse problems in BLT and UMBLT.
In the diffusive regime (that is, the light propagation is modeled by the diffusion equation), the BLT and UMBLT aim to recover the spatial distribution of the bioluminescent source, that is $S(x)$ in~\eqref{eq:DE} and $S_0(x)$ in~\eqref{eq:mDE}, respectively. 
The diffusive BLT measures a single diffusion solution at the boundary. This type of boundary data has a lower dimension compared to that of the unknown source, resulting in an underdetermined inverse problem that generally suffers from nonuniqueness unless a priori information is provided regarding the source~\cite{cong2005practical, isakov1990inverse}. 
Various strategies have been proposed in the literature to address the under-determination in BLT. One of them utilizes the idea of ultrasound modulation, leading to the development of the UMBLT.
The diffusive UMBLT measures a series of perturbed diffusion solutions at the boundary. Through asymptotic analysis and integration-by-parts techniques, this boundary data can be readily converted into equivalent internal data, resulting in a formally-determined inverse problems~\cite{bal2014ultrasound}.

In the transport regime (that is, the light propagation is modeled by the radiative transfer equation), the inverse problems in BLT and UMBLT seek to recover a bioluminescent source in the radiative transfer equation (RTE). 
The transport BLT measures angularly-resolved RTE solution at the boundary. The angular measurement provides additional information in contrast to diffusive BLT, making the transport BLT problem formally-determined ($n=2$) or even overdetermined ($n\geq 3$). In particular, some uniqueness, stability, and reconstruction results have been obtained for the transport BLT problem in~\cite{bal2007inverse,fujiwara2019fourier,fujiwara2020numerical,fujiwara2021source,stefanov2008inverse}. 
On the other hand, the transport UMBLT measures a series of perturbed RTE solutions at the boundary. This boundary data can be likewise converted into internal data, resulting in an inverse source problem with internal functional data for the RTE~\cite{BalSurvey}. Several uniqueness and stability results have been established in~\cite{bal2016ultrasound,chung2020ultrasound}

\medskip
\textbf{Contribution of the Paper.} The paper proposes a reconstructive source imaging procedure for diffusive UMBLT in optically anisotropic media with partial data and uncertain optical parameters. Within the framework of mathematical theory of diffusive UMBLT, the major contributions include:
\begin{itemize}
    \item Reconstruction in Optically Anisotropic Media. Optically anisotropic materials have different optical properties depending on the direction of light propagation within them. This is in contrast to optically isotropic materials, where the optical properties remain the same regardless of direction. A reconstruction procedure for diffusive UMBLT has been obtained in optically isotropic media~\cite{bal2014ultrasound}. In Section~\ref{sec:full}, we follow the idea of the proof in~\cite{bal2014ultrasound} and generalize it to optically anisotropic media. The study provides a more comprehensive understanding of diffusive UMBLT imaging in optically complex media. 
    \item Reconstruction with Partial Data. In practical situations, it is common to have access only to partial or incomplete measurements due to limitations in sensing devices or environmental factors. Consequently, our study extends to source imaging in diffusive UMBLT when data is solely attainable at partial boundary. Our results encompasses the refinement of the reconstruction procedure to accommodate partial data, thereby furnishing a theoretical underpinning for source imaging with limited data acquisition, see Theorem~\ref{thm:positive}.
    \item Uncertainty Quantification from the PDE Perspective. Our reconstruction procedure, with full or partial data, hinges essentially on prior knowledge of optical parameters, notably the diffusion coefficient and the absorption coefficient. As a result, it is paramount to understand the consequence of inaccuracies within these optical parameters on the source imaging process. One method to quantify such a consequence involves assessing the discrepancies between PDE solutions~\cite{lai2022inverse,ren2020characterizing}. In this paper, we take this perspective to investigate the source imaging problem in UMBLT. We derive a quantitative uncertainty estimate using the PDE theory of second-order elliptic equations, see Theorem~\ref{thm:continuousUQ}. The estimate demonstrates how the variance of the source is linked to the variance of the optical parameters.
    \item Discrete Formulation for Diffusive UMBLT. The diffusive UMBLT model is further discretized using the staggered grid scheme to yield a discrete model. This discrete formulation serves two purposes: on the one hand, it provides a finite dimensional formulation of the source imaging problem in UMBLT; on the other hand, it facilitates the subsequent numerical implementation and validation of the diffusive model. Our analysis is further extended to this discrete model: we prove that the finite-dimensional formulation is well posed, adapt the reconstructive procedure to the discrete model, and derive a discrete estimate to quantify the impact of uncertain optical parameters on the discrete source imaging process, see Theorem~\ref{thm:discreteUQ}.
\end{itemize}

\textbf{Paper Organization.} 
The paper is structured as follows. In Section~\ref{sec:full}, we derive internal data from the boundary measurement in UMBLT assuming plan-wave modulation, and propose the reconstruction procedure with full data in anisotropic media. This reconstruction procedure is generalized in Section~\ref{sec:partial} to the situation where only partial boundary measurement is available. Section~\ref{sec:continuous} establishes an uncertainty quantification estimate for the reconstruction procedure. Section~\ref{sec:discrete} discretizes the diffusion equation using the staggered grid scheme to result in a discrete formulation of the UMBLT inverse problem. A discrete reconstruction procedure is derived along with a discrete uncertainty quantification estimate. Section~\ref{sec:numerics} is devoted to the implementation of the reconstruction procedure as well as quantitative validation using numerical examples.

\bigskip
\section{Reconstruction with Full Data} \label{sec:full}

Throughout the paper, the following hypotheses are made regarding the anisotropic diffusion coefficient $D(x)$ and the absorption coefficient $\sigma_a(x)$:
\medskip
\begin{itemize}
    \item[\textbf{H1}] $D(x)$ is a matrix-valued function and $D(x)=I$ near $\partial\Omega$. Here, $I$ is the identity matrix.
    \item[\textbf{H2}] $\sigma_a\in C^\alpha(\Omega),D_{ij}\in C^{1,\alpha}(\Omega)$ where $C^{k,\alpha}$ is the H\"older space of order $k$ with exponent $\alpha\in (0,1)$.
    \item[\textbf{H3}] $D(x)$ is positive definite for all $x\in\Omega$, that is, there exists a constant $\lambda>0$ such that 
    \[\frac{1}{\lambda}|\xi|^2\geq\xi^\top D(x)\xi\geq\lambda|\xi|^2\quad\text{ a.e. on } \overline{\Omega}\] 
    holds for any $\xi\in\mathbb{R}^n$. % $\lambda$ on $\overline{\Omega}$;
    \item[\textbf{H4}] $\sigma_a \geq 0$ a.e. on $\overline{\Omega}$. 
\end{itemize}
Under these hypotheses, we will derive a reconstructive procedure to recover the internal source $S$, provided the anisotropic diffusion coefficient $D(x)$ and the absorption coefficient $\sigma_a(x)$ are given. The idea is similar to the proof in~\cite{bal2014ultrasound} in spirit, but is generalized to anisotropic $D(x)$. Recall that the full boundary measurement means $\Gamma=\partial\Omega$.

\medskip
Consider the adjoint problem to~\eqref{eq:mDE} \eqref{eq:mDEbc} with $\varepsilon=0$ and a prescribed Robin boundary condition $g$:
\begin{align}
    -\nabla\cdot D(x) \nabla \psi(x) + \sigma_{a}(x) \psi(x) & = 0\quad &\text{ in }  \Omega. \label{eq:aDE} \\
    \psi +\ell\nu\cdot D\nabla\psi & = g &\text{  on } \partial \Omega. \label{eq:aDEbc}
\end{align}
Note that the adjoint solution $\psi$ can be computed, as $D$, $\sigma_a$ and $g$ are known.
We multiply~\eqref{eq:mDE} by $\psi$, multiple~\eqref{eq:aDE} by $\phi_\varepsilon$, then integrate their difference by parts over $\Omega$ to obtain
\begin{equation} \label{eq:36}
    -\frac{1}{\ell}\int_{\partial \Omega} g \phi_\varepsilon \dif s =\int_{\Omega}(D_\varepsilon-D_0)\nabla\phi_\varepsilon\cdot\nabla\psi+(\sigma_{a,\varepsilon}-\sigma_{a,0})\phi_\varepsilon\psi-\psi S_{\varepsilon}\dif x,
\end{equation}
where the boundary integral are computed using the boundary conditions~\eqref{eq:mDEbc} \eqref{eq:aDEbc}. Expand both sides in $\varepsilon$ using~\eqref{eq:mD}-\eqref{eq:mS} and equate the $O(\varepsilon)$-terms to obtain
\begin{equation}
        -\frac{1}{\ell}\int_{\partial \Omega} g \frac{\partial\phi_\varepsilon}{\partial\varepsilon}|_{\varepsilon=0} \dif s
        =\int_{\Omega} \left[ (2\gamma-1)D\nabla\phi_0\cdot\nabla\psi+(2\gamma+1)\sigma_{a}\phi_0\psi-\psi S)  \right] \cos(q\cdot x+\varphi)\dif x.
\end{equation}
As the left hand side is known from the measurement~\eqref{eq:DEmeas}, so is the right hand side. By varying the modulation parameters $q$ and $\varphi$, one can recover the Fourier transform of the following function:
\begin{equation}\label{eq:Hpsi}
    H_\psi\coloneqq(2\gamma-1)D\nabla\phi_0\cdot\nabla\psi+(2\gamma+1)\sigma_{a}\phi_0\psi-\psi S.
\end{equation}
If we choose a specific adjoint solution $\psi_0$ such that $\psi_0\geq c>0$ for some constant $c$, then dividing both sides by $\psi_0$ and substituting $S$ by the equation~\eqref{eq:mDE} with $\varepsilon=0$ give the following PDE
\begin{equation}\label{eq:38}
    F_{\psi_0}\coloneqq\frac{H_{\psi_0}}{\psi_0} = \nabla\cdot D \nabla \phi_0 + (2\gamma-1)D\nabla\phi_0\cdot\nabla\log\psi_0+2\gamma\sigma_{a}\phi_0.
\end{equation}
This is a second order elliptic PDE for $\phi_0$ with known coefficients, which can be solved along with the boundary condition~\eqref{eq:mDEbc} with $\varepsilon=0$ to yield $\phi_0$. Finally, the source $S$ can be computed from~\eqref{eq:DE}.

It remains to show the existence of the positive adjoint solution $\psi_0$. To see this, note that there are suitable Dirichlet boundary conditions such that a positive solution $\psi_0\geq c>0$ exists by the maximum principle. One can take the corresponding Robin data $g=\psi_0+\ell\nu\cdot D\nabla\psi_0$ to ensure the solution of~\eqref{eq:aDE} \eqref{eq:aDEbc} is $\psi_0$.

\bigskip
\section{Reconstruction with Partial Data} \label{sec:partial}

In this section, we aim to extend the reconstruction procedure in Section~\ref{sec:full} to the partial data case where the boundary measurement is made only on an open subset $\Gamma\subsetneq\partial\Omega$.
A careful examination of the proof suggests that the following modifications are necessary in order to adapt the idea: 
(1). the left hand side of~\eqref{eq:36} must be computable in order to obtain the internal data $H_\psi$ from the right hand side. In the partial data case, $\phi_\varepsilon$ is known only on $\Gamma$, this restriction requires the choice of the adjoint boundary condition $g$ to vanish on $\partial\Omega\backslash\Gamma$, that is, $g|_{\partial\Omega\backslash\Gamma}=0$. 
(2). A critical ingredient in the proof with full data is the existence of a positive adjoint solution $\psi_0>0$. In the partial data case, we need to show the existence of a positive adjoint solution $\psi_0>0$ with the additional constraint $g|_{\partial\Omega\backslash\Gamma}=0$.
Once the second modification is verified, the reconstructive procedure in Section~\ref{sec:full} would apply to the partial data case as well.

The main part of this section is devoted to proving the existence of a positive solution to the adjoint problem~\eqref{eq:aDE} \eqref{eq:aDEbc} with $g|_{\partial\Omega\backslash\Gamma}=0$.
Instead of directly constructing a positive adjoint solution, we consider the following adjoint equation with mixed boundary conditions:
\begin{alignat}{2}
    -\nabla\cdot D(x) \nabla \psi(x) + \sigma_{a}(x) \psi(x) & = 0\quad &&\text{ in }  \Omega. \label{eq:mixDE} \\
    \psi +\ell\nu\cdot D\nabla\psi & = 0 &&\text{ on } \partial \Omega \setminus \Gamma . \label{eq:mixDEbc1}\\
    \psi & = f &&\text{ on } \Gamma . \label{eq:mixDEbc2}
\end{alignat}
Once we find a positive solution $\psi$ to this mixed boundary value problem, we can simply take $g=(\psi +\ell\nu\cdot D\nabla\psi)|_{\partial\Omega}$ in the adjoint problem~\eqref{eq:aDE}-\eqref{eq:aDEbc}, then the adjoint solution is $\psi>0$. 

The following result ensures the well-posedness of the mixed boundary value problem.

\begin{prop}[{\cite[Theorem 1]{lieberman1986mixed}}]\label{thm:mixWP}
Assume that
\[\sigma_a\in C^\alpha(\Omega),\qquad D_{ij}\in C^{1,\alpha}(\Omega),\qquad f\in C(\Gamma)\cap L^\infty(\Gamma),\]
then \eqref{eq:mixDE}-\eqref{eq:mixDEbc2} has a unique solution $\psi\in C^2(\overline{\Omega}\setminus\overline{\Gamma})\cap C^0(\overline{\Omega})$
\end{prop}

\begin{thm}\label{thm:positive}
Supppose the hypotheses \textbf{H1}-\textbf{H4} hold. If the Dirichlet boundary condition $f\in C(\Gamma)\cap L^\infty(\Gamma)$ is positive, then the mixed boundary value problem \eqref{eq:mixDE}-\eqref{eq:mixDEbc2} admits a unique solution $\psi\in C^2(\overline{\Omega}\setminus\overline{\Gamma})\cap C^0(\overline{\Omega})$ which is positive on $\overline{\Omega}$.
\end{thm}
\begin{proof}
    By Proposition~\ref{thm:mixWP}, the mixed boundary value problem has a unique solution $\psi\in C^2(\overline{\Omega}\setminus\overline{\Gamma})\cap C^0(\overline{\Omega})$. 
    Suppose $\psi$ takes negative values on $\overline{\Omega}$, the weak maximum principle \cite[Section 6.4 Theorem 2]{evans10} 
    claims that the minimum is achieved on the boundary $\partial\Omega$. Since $\psi|_\Gamma >0$, the minimum must be achieved at a point $x_0\in\partial\Omega\setminus\Gamma $, that is, $\psi(x_0)=\inf_{x\in\overline{\Omega}}\psi<0$. According to the Robin boundary condition \eqref{eq:mixDEbc1}, we have
    $$
    \partial_\nu\psi(x_0) = \nu\cdot D\nabla\psi(x_0) = -\frac{1}{\ell} \psi(x_0) >0
    $$
    where the first equality holds since $D(x) = I$ near $\partial\Omega$. This contradicts that $x_0$ is a global minimum of $\psi$ over $\overline{\Omega}$. Therefore, $\psi\geq 0$ on $\overline{\Omega}$.

    If $\psi$ achieves the zero minimum at an interior point, that is, $\psi(x)=0$ for some $x\in \Omega$, the strong maximum principle \cite[Section 6.4 Theorem 4]{evans10} forces $\psi\equiv {\rm constant}$ in $\Omega$. In view of the Robin boundary condition on $\partial\Omega\backslash\Gamma$, we have $\psi\equiv 0$, contradicting that $\psi|_{\Gamma}=f>0$. Therefore, $\psi>0$ in $\Omega$.

    It remains to show $\psi|_{\partial\Omega} >0$, or more precisely, $\psi|_{\partial\Omega\backslash\Gamma}>0$ since $\psi|_{\Gamma} = f>0$. Suppose otherwise, that is, there exists $x_0\in\partial\Omega\setminus\Gamma $ such that $\psi(x_0)=\inf_{x\in\overline{\Omega}}\psi=0$. Applying the Hopf Lemma \cite[Section 6.4 Lemma 3(ii)]{evans10} to $-\psi$ shows that $\partial_\nu\psi(x_0)<0$, then 
    $$
    \psi(x_0)+\ell\nu\cdot D\nabla\psi(x_0) = \ell \partial_\nu  \psi(x_0) < 0,
    $$
    contraciting the boundary condition on $\partial\Omega\backslash\Gamma$. Therefore, we must have $\psi|_{\partial\Omega\backslash\Gamma}>0$.

    Combining all the cases, we see that $\psi$ is a positive solution on the compact set $\overline{\Omega}$, hence has a positive lower bound. This completes the proof.
\end{proof}

\begin{remark}
Theorem~\ref{thm:positive} ensures the existence of a positive adjoint solution $\psi>0$ with partial data, then we can reconstruct the source $S$ using the same process as for the full data case.    
\end{remark}

\bigskip
\section{Uncertainty Quantification with Continuous Diffusive Model} \label{sec:continuous}

The reconstructive procedures in Section 2 and Section 3 rely essentially on accurate prior knowledge of the optical coefficients $(D,\sigma)$ to solve the elliptic equation~\eqref{eq:38} (along with boundary conditions) for $\phi_0$. The underlying rationale is that these optical coefficients can be measured in advance using other imaging modalities such as optical tomography~\cite{arridge2009optical}. Practically, the imaging process in these additional modalities inevitably introduces inaccuracy to the optical coefficients, which in turn will impact the UMBLT reconstructions. 
In the subsequent two sections, we aim to quantify the impact to the reconstruction of the bio-luminescence source $S$ that is due to the inaccuracy of the optical coefficients, using the continuous and discretized models respectively.

\medskip
Let $(D,\sigma_a)$ be the underlying true optical coefficients, and $(\tilde{D},\tilde{\sigma}_a)$ be the optical coefficients that are reconstructed through additional imaging modalities before performing UMBLT. 
Observe that $(\tilde{D},\tilde{\sigma}_a)$ do not play a role in the derivation of the internal data: This is because the boundary integral on the left hand side of~\eqref{eq:36} remains the same, thus we can derive $H_\psi$ as before. Hereafter, we will assume the internal data $H_\psi$ has been accurately extracted, and focus on quantifying the uncertainty of the reconstructed source $S$. The full data case and partial data case will be handled in one shot, since the reconstruction process are identical once a suitable positive adjoint solution $\psi_0 >0$ is chosen.

\medskip
We record a regularity result for the diffusion equation with Robin boundary conditions. Here, $W^{s,\infty}(\Omega)$ and $H^{s}(\Omega)$ denotes the $L^\infty$-based and $L^2$-based Sobolev spaces of order $s\in\mathbb{R}$, respectively.
%\Yang{\st{Do our hypotheses satisfy all the assumptions in \cite[Theorem~2.4]{dong2022w}, especially Asumption 2.2 there?}}

\begin{prop}[{\cite[Theorem~2.4]{dong2022w}}] \label{thm:ellipreg}
Suppose $D$ is uniformly elliptic, $D_{ij}\in W^{1,\infty}(\Omega)$, $\sigma_a\geq0$ a.e. For $S\in L^2(\Omega)$ and $g\in H^{\frac{1}{2}}(\partial\Omega)$, the following boundary value problem 
\begin{alignat}{2} 
    -\nabla\cdot D(x) \nabla \phi(x) + \sigma_{a}(x) \phi(x) & = S(x)\quad &&\text{ in } \Omega. \label{eq:30}\\
    \phi +\ell\nu\cdot D\nabla\phi & = g &&\text{ on } \partial \Omega.\label{eq:31}
\end{alignat}
admits a unique solution $\phi\in H^2(\Omega)$ with the estimate
\begin{equation}
    \|\phi\|_{H^2(\Omega)}\leq C(\|S\|_{L^2(\Omega)}+\|g\|_{H^{\frac{1}{2}}(\partial\Omega)})
\end{equation}
where $C$ is a constant independent of $\phi$.
\end{prop}

We have the following global uncertainty quantification (UQ) estimate for the diffusive UMBLT reconstruction.

\begin{thm} \label{thm:continuousUQ}
Suppose all optical coefficients and solutions satisfy
\begin{align*}
&\|D_{ij}\|_{W^{1,\infty}(\Omega)},\|\tilde D_{ij}\|_{W^{1,\infty}(\Omega)} \leq C_D,&&\|\phi\|_{W^{2,\infty}(\Omega)},\|\tilde\phi\|_{W^{2,\infty}(\Omega)} \leq C_\phi,\\
&\|\psi\|_{W^{2,\infty}(\Omega)},\|\tilde{\psi}\|_{W^{2,\infty}(\Omega)} \leq C_\psi,&&\|\sigma_a\|_{L^\infty(\Omega)} \leq C_\sigma,\\
&\psi,\tilde\psi\geq c_\psi>0,&&
\end{align*}
where $C_D,C_\phi,C_\psi,C_\sigma,c_\psi$ are constants, and $0$ is not eigenvalue of the following operators equipped with the zero Robin boundary condition:
\[\nabla\cdot D \nabla + (2\gamma-1)D\nabla\log\psi_0\cdot\nabla+2\gamma\sigma_{a},\quad\nabla\cdot \tilde{D} \nabla  + (2\gamma-1)\tilde{D}\nabla\log\tilde{\psi}_0\cdot\nabla + 2\gamma\tilde{\sigma}_{a},\]
then we can find constants $C_{1ij},C_2>0$ such that
\begin{equation} \label{eq:Sest}
    \|S-\tilde S\|_{L^2(\Omega)}\leq\sum_{i\leq j} C_{1ij}\|(D-\tilde D)_{ij}\|_{H^1(\Omega)} + C_2\|\sigma_a-\tilde \sigma_a\|_{L^2(\Omega)}
\end{equation}
\end{thm}
\begin{proof}
Let $\phi$ and $\tilde\phi$ solve the diffusion equations
$$
S = -\nabla\cdot [D\nabla \phi]+\sigma_a \phi, \quad\quad\quad 
\tilde S = -\nabla\cdot[\tilde D\nabla\tilde \phi]+\tilde\sigma_a \tilde \phi,
$$
respectively. Subtract these equations to get
$$    
S - \tilde S = -\nabla \cdot[(D-\tilde D)\nabla \phi]-\nabla\cdot[\tilde D\nabla (\phi-\tilde \phi)]+(\sigma_a-\tilde\sigma_a) \phi+\tilde\sigma_a(\phi-\tilde \phi).
$$
Taking the $L^2$-norms on both sides, we have
\begin{equation} \label{eq:Sdiff}
    \begin{aligned}
    &\|S-\tilde S\|_{L^2(\Omega)}\\
    \leq & \ \|\nabla \cdot[(D-\tilde D)\nabla \phi]\|_{L^2(\Omega)}+\|\nabla\cdot[\tilde D\nabla (\phi-\tilde \phi)]\|_{L^2(\Omega)} + \|(\sigma_a-\tilde\sigma_a) \phi\|_{L^2(\Omega)}+\|\tilde\sigma_a(\phi-\tilde \phi)\|_{L^2(\Omega)}\\
    \leq&\sum_{ij}\|\partial_j\phi\|_{L^\infty(\Omega)}\|\partial_i (D-\tilde D)_{ij}\|_{L^2(\Omega)}+\sum_{ij}\|\partial_{ij}\phi\|_{L^\infty(\Omega)}\|(D-\tilde D)_{ij}\|_{L^2(\Omega)}\\
    &+\sum_{ij}\|\partial_i\tilde D_{ij}\|_{L^\infty(\Omega)}\|\partial_j (\phi-\tilde\phi)\|_{L^2(\Omega)}+\sum_{ij}\|\tilde D_{ij}\|_{L^\infty(\Omega)}\|\partial_{ij} (\phi-\tilde\phi)\|_{L^2(\Omega)}\\
    &+\|\phi\|_{L^\infty(\Omega)}\|\sigma_a-\tilde\sigma_a\|_{L^2(\Omega)}+\|\tilde\sigma_a\|_{L^\infty(\Omega)}\|\phi-\tilde \phi\|_{L^2(\Omega)}\\
    \leq& c_1\|\phi-\tilde \phi\|_{H^2(\Omega)}+\sum_{i\leq j}c_{2ij}\|(D-\tilde D)_{ij}\|_{H^1(\Omega)}+c_3\|\sigma_a-\tilde\sigma_a\|_{L^2(\Omega)}
    \end{aligned}
\end{equation}
where the constants $c_1,c_{2ij},c_3>0$ can be made explicit as follows:
{\small
\begin{align*}
    c_1&=\sqrt{\|\tilde\sigma_a\|_{L^\infty(\Omega)}^2+\sum_j\left[\sum_{i}\|\partial_i\tilde D_{ij}\|_{L^\infty(\Omega)}\right]^2+4\sum_{i<j}\|\tilde D_{ij}\|_{L^\infty(\Omega)}^2+\sum_{i}\|\tilde D_{ii}\|_{L^\infty(\Omega)}^2}\\
    c_{2ij}&=\sqrt{4\|\partial_{ij}\phi\|_{L^\infty(\Omega)}^2+\left(\|\partial_i\phi\|_{L^\infty(\Omega)}+\|\partial_j\phi\|_{L^\infty(\Omega)}\right)^2}\qquad(i<j)\\
    c_{2ii}&=\sqrt{\|\partial_{ii}\phi\|_{L^\infty(\Omega)}^2+\|\partial_i\phi\|_{L^\infty(\Omega)}^2}\\
    c_3&=\|\phi\|_{L^\infty(\Omega)}
\end{align*}
}

In order to estimate the term $\|\phi-\tilde \phi\|_{H^2(\Omega)}$, we turn to the second order elliptic equations generated from the internal data $H_{\psi} = H_{\tilde\psi}$:
\begin{align*}
F_{\psi}&=\frac{H_{\psi}}{\psi}=(2\gamma-1)D\nabla\phi\cdot\nabla\log\psi+2\gamma\sigma_{a}\phi+\nabla\cdot D \nabla \phi\\
    F_{\tilde{\psi}} &= \frac{H_{\psi}}{\tilde \psi}=(2\gamma-1)\tilde D\nabla\tilde\phi\cdot\nabla\log\tilde\psi+2\gamma\tilde\sigma_{a}\tilde\phi+\nabla\cdot \tilde D \nabla \tilde\phi.
\end{align*}
Subtracting these equations gives
\begin{equation*}
    \begin{aligned}
    &-\nabla\cdot \tilde D \nabla [\phi-\tilde\phi] -2\gamma\tilde{\sigma}_a(\phi-\tilde{\phi})-(2\gamma-1)\tilde D\nabla (\phi-\tilde \phi)\cdot\nabla\log\psi\\
    & = \frac{H_{\psi}}{\psi\tilde\psi}(\psi-\tilde\psi)+(2\gamma-1)(D-\tilde D)\nabla \phi\cdot\nabla\log \psi\\
    &+(2\gamma-1)\tilde D\nabla \tilde\phi\cdot(\nabla\log \psi-\nabla\log\tilde \psi)+2\gamma(\sigma_a-\tilde{\sigma}_a)\phi+\nabla\cdot [D-\tilde D] \nabla\phi,
    \end{aligned}
\end{equation*}
This is a second order elliptic equation for $\phi-\tilde \phi$ with zero Robin boundary condition, we have the following regularity estimate by Proposition~\ref{thm:ellipreg}:
\begin{equation} \label{eq:phidiff}
    \begin{aligned}
    & \|\phi-\tilde \phi\|_{H^2(\Omega)}\\
    \leq &
    C \biggl(\left\|\frac{H_{\psi}}{\psi\tilde\psi}(\psi-\tilde\psi)\right\|_{L^2(\Omega)} + |2\gamma-1|\|(D-\tilde D)\nabla \phi\cdot\nabla\log \psi\|_{L^2(\Omega)} \\
    & +|2\gamma-1|\|\tilde D\nabla \tilde\phi\cdot(\nabla\log \psi-\nabla\log\tilde \psi)\|_{L^2(\Omega)} +\|\nabla \cdot[D-\tilde D]\nabla \phi\|_{L^2(\Omega)}+|2\gamma|\|(\sigma_a-\tilde{\sigma}_a)\phi\|_{L^2(\Omega)}\biggr)\\
    \leq & C\biggl(\frac{\|H_{\psi}\|_{L^\infty(\Omega)}}{c_\psi^2}\|\psi-\tilde\psi\|_{L^2(\Omega)} + |2\gamma-1|\sum_{ij}\|\partial_i\log \psi\|_{L^\infty(\Omega)}\|\partial_j \phi\|_{L^\infty(\Omega)}\|(D-\tilde D)_{ij}\|_{L^2(\Omega)}\\
    &+|2\gamma-1|\sum_{ij}\|\tilde D_{ij}\|_{L^\infty(\Omega)}\|\partial_j \tilde\phi\|_{L^\infty(\Omega)}\|\partial_i(\log \psi-\log\tilde \psi)\|_{L^2(\Omega)} + \sum_{ij}\|\partial_j\phi\|_{L^\infty(\Omega)}\|\partial_i (D-\tilde D)_{ij}\|_{L^2(\Omega)} \\
    & +\sum_{ij}\|\partial_{ij}\phi\|_{L^\infty(\Omega)}\|(D-\tilde D)_{ij}\|_{L^2(\Omega)} + |2\gamma|\|\phi\|_{L^\infty(\Omega)}\|\sigma_a-\tilde{\sigma}_a\|_{L^2(\Omega)}\biggr)\\
    \leq& c_4\|\psi-\tilde \psi\|_{H^1(\Omega)}+\sum_{i\leq j}c_{5ij}\|(D-\tilde D)_{ij}\|_{H^1(\Omega)}+c_6\|\sigma_a-\tilde{\sigma}_a\|_{L^2(\Omega)}
    \end{aligned}
\end{equation}
where in the last inequality, we used the upper bound $\|\partial_i\log \psi\|_{L^\infty(\Omega)}\leq\frac{1}{c_\psi}\|\partial_i\psi\|_{L^\infty(\Omega)}$ and
 \begin{equation*}
      \begin{aligned}
     \|\partial_i(\log \psi-\log\tilde \psi)\|_{L^2(\Omega)}
     \leq & \frac{1}{c_\psi^2}\|\psi \partial_i\tilde \psi-\tilde\psi \partial_i\psi\|_{L^2(\Omega)}
     = \frac{1}{c_\psi^2}\|(\psi-\tilde\psi) \partial_i\tilde \psi-\tilde\psi \partial_i(\psi-\tilde{\psi})\|_{L^2(\Omega)}\\
     \leq & \frac{1}{c_\psi^2}\|\partial_i\tilde \psi\|_{L^\infty(\Omega)}\|\psi-\tilde\psi\|_{L^2(\Omega)} +\frac{1}{c_\psi^2}\|\tilde\psi\|_{L^\infty(\Omega)}\| \partial_i(\psi-\tilde{\psi})\|_{L^2(\Omega)}
 \end{aligned}
 \end{equation*}
The constants $c_4,c_{5ij},c_6>0$ are defeind as
{\small
 \begin{align*}
      c_4=&\frac{C|2\gamma-1|}{c_\psi^2}\left(\sum_{i}\left(\sum_j\|\tilde D_{ij}\|_{L^2(\overline{\Omega})}\|\partial_j \phi\|_{L^\infty(\overline{\Omega})}\|\tilde\psi\|_{L^\infty(\overline{\Omega})}\right)^2\right. +\left.\left(\frac{\|H_{\psi}\|_{L^\infty(\overline{\Omega})}}{|2\gamma-1|}+\sum_{ij}\|\tilde D_{ij}\|_{L^2(\overline{\Omega})}\|\partial_j \phi\|_{L^\infty(\overline{\Omega})}\|\partial_i\tilde\psi\|_{L^\infty(\overline{\Omega})}\right)^2\right)^{\frac{1}{2}} \\
     c_{5ij}=&C\cdot\Biggl(\left(2\|\partial_{ij}\phi\|_{L^\infty(\overline{\Omega})}+\frac{|2\gamma-1|}{c_\psi}\|\partial_i\psi\|_{L^\infty(\overline{\Omega})}\|\partial_j \phi\|_{L^\infty(\overline{\Omega})}+\frac{|2\gamma-1|}{c_\psi}\|\partial_j\psi\|_{L^\infty(\overline{\Omega})}\|\partial_i \phi\|_{L^\infty(\overline{\Omega})}\right)^2\\
     &\qquad\qquad\qquad+\left(\|\partial_i \phi\|_{L^\infty(\overline{\Omega})}+\|\partial_j \phi\|_{L^\infty(\overline{\Omega})}\right)^2\Biggr)^{\frac{1}{2}}\qquad(i<j)\\
     c_{5ii}=&C\cdot\sqrt{\left(\|\partial_{ii}\phi\|_{L^\infty(\overline{\Omega})}+\frac{|2\gamma-1|}{c_\psi}\|\partial_i\psi\|_{L^\infty(\overline{\Omega})}\|\partial_i \phi\|_{L^\infty(\overline{\Omega})}\right)^2+\|\partial_i \phi\|_{L^\infty(\overline{\Omega})}^2}\\
     c_6=&|2\gamma|C\cdot\|\phi\|_{L^\infty(\overline{\Omega})}
 \end{align*}
}

It remains to estimate the term $\|\psi-\tilde \psi\|_{H^1(\Omega)}$. Let us consider the adjoint equations
\begin{equation}
    \begin{aligned}
    -\nabla\cdot D\nabla \psi+\sigma_a \psi &= 0,\\
    -\nabla\cdot\tilde D\nabla\tilde \psi+\tilde\sigma_a \tilde \psi &= 0.
    \end{aligned}
\end{equation}
Subtract these two equations to get
\begin{equation}
    -\nabla\cdot\tilde D\nabla( \psi-\tilde \psi)+\tilde\sigma_a(\psi-\tilde \psi)=\nabla\cdot (D-\tilde D)\nabla \psi-(\sigma_a-\tilde\sigma_a) \psi
\end{equation}
This is a second order elliptic equation for $\psi-\tilde \psi$ with the zero Robin boundary condition. Again, by the elliptic regularity result Proposition~\ref{thm:ellipreg}, we have
\begin{equation} \label{eq:psidiff}
    \begin{aligned}
    &\|\psi-\tilde \psi\|_{H^1(\Omega)}\\
    \leq& C(\|\nabla\cdot [(D-\tilde D)\nabla \psi]\|_{L^2(\Omega)}+\|(\sigma_a-\tilde\sigma_a) \psi\|_{L^2(\Omega)})\\
    \leq& C\biggl(\sum_{ij}\|\partial_j\psi\|_{L^\infty(\Omega)}\|\partial_i (D-\tilde D)_{ij}\|_{L^2(\Omega)}+\sum_{ij}\|\partial_{ij}\psi\|_{L^\infty(\Omega)}\|(D-\tilde D)_{ij}\|_{L^2(\Omega)} + \|\psi\|_{L^\infty(\Omega)}\|(\sigma_a-\tilde\sigma_a) \|_{L^2(\Omega)}\biggr)\\
    \leq& \sum_{i\leq j}c_{7ij}\|(D-\tilde D)_{ij}\|_{H^1(\Omega)}+c_8\|\sigma_a-\tilde \sigma_a\|_{L^2(\Omega)}
    \end{aligned}
\end{equation}
with constants $c_{7ij},c_8>0$, where
\begin{align*}
    c_{7ij}&=C\cdot\sqrt{\left(\|\partial_i\psi\|_{L^\infty(\Omega)}+\|\partial_j\psi\|_{L^\infty(\Omega)}\right)^2+4\|\partial_{ij}\psi\|_{L^\infty(\Omega)}^2}\qquad(i<j)\\
    c_{7ii}&=C\cdot\sqrt{\|\partial_i\psi\|_{L^\infty(\Omega)}^2+\|\partial_{ii}\psi\|_{L^\infty(\Omega)}^2}\\
    c_8&=C\cdot\|\psi\|_{L^\infty(\Omega)}
\end{align*}
Combining~\eqref{eq:Sdiff} \eqref{eq:phidiff} \eqref{eq:psidiff}, we conclude that
\begin{equation}
    \|S-\tilde S\|_{L^2(\Omega)}\leq \sum_{i\leq j}C_{1ij}\|(D-\tilde D)_{ij}\|_{H^1(\Omega)} + C_2\|\sigma-\tilde \sigma\|_{L^2(\Omega)},
\end{equation}
with $C_{1ij}=c_1c_4c_{7ij}+c_1c_{5ij}+c_{2ij}$ and $C_2=c_1c_4c_8+c_1c_6+c_3$. Note that all the constants in this proof are explicit, except for the constant $C$ that comes from the estimate of elliptic regularity.
\end{proof}

\begin{remark}
Theorem~\ref{thm:continuousUQ} can be interpreted as follows. Squaring the estimate~\eqref{eq:Sest} gives
$$
\|S-\tilde S\|^2_{L^2(\Omega)}\leq \mathfrak{C} \left( \|D-\tilde D\|^2_{H^1(\Omega)} + \|\sigma_a-\tilde \sigma_a\|^2_{L^2(\Omega)} \right)
$$
where the constant $\mathfrak{C}$ is in terms of $C_{1ij}$ and $C_2$. If we take $S, D, \sigma_a$ to be the underlying ground-truth parameters and $\tilde{S}, \tilde{D}, \tilde{\sigma}_a$  the corresponding parameters in the presence of additive random uncertainty of mean zero, then $\mathbb{E}[\tilde{S}] = S$, $\mathbb{E}[\tilde{D}] = D$, $\mathbb{E}[\tilde{\sigma}_a] = \sigma_a$. The estimate provides a quantitative error bound on the variance of the bioluminescent source. 
\end{remark}

\bigskip
\section{Uncertainty Quantification with Discretized Diffusive Model} \label{sec:discrete}

In the previous section, we considered the impact of inaccurate $(D,\sigma_a)$ using continuous PDE models. However, for the subsequent numerical simulation, the PDEs have to be discretized into finite dimensional discrete models. This motivates us to study a similar UQ problem based on the finite difference discretization of the PDE model. The analysis in this section provides a finite dimensional counterpart of the infinite dimensional UQ estimate~\eqref{eq:Sest}, bridging the gap between the infinite dimensional analysis and the finite dimensional numerical experiments.

We will consider the discretization of three diffusion-type equations: the forward problem~\eqref{eq:mDE} \eqref{eq:mDEbc}, the adjoint problem~\eqref{eq:aDE} \eqref{eq:aDEbc}, and the internal data problem~\eqref{eq:Hpsi} equipped with the zero Robin boundary condition. These problems need to be discretized in order to implement the reconstruction procedure outlined in Section~\ref{sec:full}. The discretization procedure requires numerical evaluation of the terms $\nabla\cdot D\nabla\phi_0$, $D\nabla\phi_0\cdot\nabla \log\psi_0$, and $\sigma_a\phi_0$. The last term can be readily evaluated on a grid. In the following, we explain how to discretize the first two differential operators using the staggered grid scheme.

We take $\Omega$ to be a 2D domain to agree with the setup of the subsequent numerical experiments. The 2D coordinates are written as $(x,y)$. The problem in 3D can be considered likewise with an additional spatial variable. Let $\Delta x$, $\Delta y$ denote the grid size on the $x$-direction and $y$-direction, respectively. 
We will discretize the divergence-form diffusion operator using the staggered grid scheme, see Figure~\ref{fig:staggered_grid}. The black dots are indexed by $(i,j)$, where $i=1,2,\dots,N_x$, $j=1,2,\dots,N_y$, white dots are indexed by $(i+\frac{1}{2},j)$, where $i=1,2,\dots,N_x-1$, $j=1,2,\dots,N_y$ and $(i,j+\frac{1}{2})$, where $i=1,2,\dots,N_x$, $j=1,2,\dots,N_y-1$. For a function $u$, we use $u_{i,j}$ to represent an approximate value of $u(x_i,y_j)$, where $x_i = x_1 + (i-1) \Delta x$ and $y_j = y_1 + (j-1) \Delta y$ are the coordinates of the grid points.

\begin{figure}[h]
    \centering
    \includegraphics[width = 0.3\textwidth]{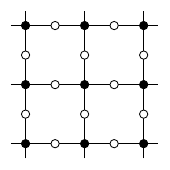}
    \caption{The illustration of staggered grid scheme. The zero and second order terms are defined on the grid points (black dots), the first order terms and $D$ are defined on the edges (while dots).}
    \label{fig:staggered_grid}
\end{figure}

\subsection{Discretization with Isotropic Diffusion Coefficients.}
We begin the discretization with an isotropic diffusion coefficient, that is, $D=D(x)$ is a scalar function.

\subsubsection{Discretization of the Forward Problem.}
First, we consider discretization of the forward problem~\eqref{eq:mDE} \eqref{eq:mDEbc}. Using the staggered grid scheme, the operator $\nabla\cdot D\nabla$ is discretized as 
\begin{equation}\label{eq:stagger1}
    \begin{aligned}
        [\nabla\cdot D\nabla u]_{i,j} 
        =& [\partial_xD\partial_xu + \partial_yD\partial_yu]_{i,j}\\
        \approx & \frac{[D\partial_xu]_{i+\frac{1}{2},j} - [D\partial_xu]_{i-\frac{1}{2},j}}{\Delta x} + \frac{[D\partial_yu]_{i,j+\frac{1}{2}}-[D\partial_yu]_{i,j-\frac{1}{2}}}{\Delta y}\\
        \approx & \frac{D_{i+\frac{1}{2},j}[u_{i+1,j}-u_{i,j}] - D_{i-\frac{1}{2},j}[u_{i,j}-u_{i-1,j}]}{\Delta x^2} \\
         &+ \frac{D_{i,j+\frac{1}{2}}[u_{i,j+1}-u_{i,j}]-D_{i,j-\frac{1}{2}}[u_{i,j}-u_{i,j-1}]}{\Delta y^2}\\
        =& \left[\frac{D_{i+\frac{1}{2},j}}{\Delta x^2}\right]u_{i+1,j} + \left[\frac{D_{i-\frac{1}{2},j}}{\Delta x^2}\right]u_{i-1,j} + \left[\frac{D_{i,j+\frac{1}{2}}}{\Delta y^2}\right]u_{i,j+1} + \left[\frac{D_{i,j-\frac{1}{2}}}{\Delta y^2}\right]u_{i,j-1}  \\
         &-\left[\frac{D_{i+\frac{1}{2},j}}{\Delta x^2}+\frac{D_{i-\frac{1}{2},j}}{\Delta x^2}+\frac{D_{i,j+\frac{1}{2}}}{\Delta y^2}+\frac{D_{i,j-\frac{1}{2}}}{\Delta y^2}\right]u_{i,j},
    \end{aligned}    
\end{equation}
where $\approx$ denotes the staggered grid scheme approximation.

For the Robin boundary condition on the four boundaries (excluding the four corners), it is simply $u\pm2D\partial_xu$ on the right/left boundary, $u\pm2D\partial_yu$ on the top/bottom boundary. For the four corner points, e.g. the bottom left corner (Figure~\ref{fig:staggered_grid_boundary}), the outgoing vector $\nu$ is chosen as $(-\frac{\sqrt{2}}{2},-\frac{\sqrt{2}}{2})$. For example,
\begin{equation}\label{eq:stagger2}
    \begin{aligned}
        [u+\ell\nu\cdot D\nabla u]_{1,1} 
        =& u_{1,1} - \frac{\sqrt{2}\ell}{2}[D\partial_xu]_{1+\frac{1}{2},1} - \frac{\sqrt{2}\ell}{2}[D\partial_yu]_{1,1+\frac{1}{2}}\\
        =& u_{1,1} + \frac{\sqrt{2}\ell}{2}\frac{D_{1+\frac{1}{2},1}}{\Delta x}[u_{1,1}-u_{1,2}] + \frac{\sqrt{2}\ell}{2}\frac{D_{1,1+\frac{1}{2}}}{\Delta y}[u_{1,1}-u_{2,1}]\\
        =& \left[1+\frac{\sqrt{2}\ell D_{1+\frac{1}{2},1}}{2\Delta x}+\frac{\sqrt{2}\ell D_{1,1+\frac{1}{2}}}{2\Delta y}\right]u_{1,1} -\frac{\sqrt{2}\ell D_{1+\frac{1}{2},1}}{2\Delta x}u_{1,2} - \frac{\sqrt{2}\ell D_{1,1+\frac{1}{2}}}{2\Delta y}u_{2,1}.
    \end{aligned}    
\end{equation}

\begin{figure}[h]
    \centering
    \includegraphics[width = 0.3\textwidth]{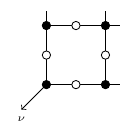}
    \caption{The outgoing vector at the corner}
    \label{fig:staggered_grid_boundary}
\end{figure}

This discretization gives rise to a linear system with the unknowns $u_{i,j}$. In order to make this linear system explicit, we introduce the index function $\mathcal{I}(i,j) \coloneqq (i-1)N_y+j$ and use $(i,j)\sim(i',j')$ to mean that the $(i',j')$-point is a neighbor of $(i,j)$-point. Denote by $I$ the set of interior points, by $B$ the set of non-corner boundary points, and by $B_c$ the set of four corner points.
According to the scheme~\eqref{eq:stagger1}, \eqref{eq:stagger2}, the forward problem~\eqref{eq:mDE} \eqref{eq:mDEbc} is discretized to yield the linear system
\[\mathbf{L}\bm{\phi}_0 = \mathbf{s}\]
where $\bm{\phi}_0$ consists of the vectorized values of the forward solution $\phi_0$ at black dots such that ${\bm{\phi}_0}_{\mathcal{I}(i,j)} = \phi_0(x_i,y_j)$.
\begin{equation}\label{eq:Lmatrix}
    \mathbf{L}_{\mathcal{I}(i,j),\mathcal{I}(i',j')} = 
    \begin{cases}
        \sum_{(\tilde{i},\tilde{j})\sim(i,j)}\frac{D_{\frac{i+\tilde{i}}{2},\frac{j+\tilde{j}}{2}}}{|i-\tilde{i}|\Delta x^2 + |j-\tilde{j}|\Delta y^2} + \sigma_{i,j},&(i',j') = (i,j),\,(i,j)\in I\\
        -\frac{D_{\frac{i+i'}{2},\frac{j+j'}{2}}}{|i-i'|\Delta x^2 + |j-j'|\Delta y^2},&(i',j') \sim (i,j),\,(i,j)\in I\\
        1+\ell\sum_{I\ni(\tilde{i},\tilde{j})\sim(i,j)}\frac{D_{\frac{i+\tilde{i}}{2},\frac{j+\tilde{j}}{2}}}{|i-\tilde{i}|\Delta x + |j-\tilde{j}|\Delta y},& (i',j') = (i,j),\,(i,j)\in B\\
        -\ell\frac{D_{\frac{i+i'}{2},\frac{j+j'}{2}}}{|i-i'|\Delta x + |j-j'|\Delta y},& I\ni(i',j') \sim (i,j)\in B,\\
        1+\frac{\sqrt{2}\ell}{2}\sum_{(\tilde{i},\tilde{j})\sim(i,j)}\frac{D_{\frac{i+\tilde{i}}{2},\frac{j+\tilde{j}}{2}}}{|i-\tilde{i}|\Delta x + |j-\tilde{j}|\Delta y},& (i',j') = (i,j),\,(i,j)\in B_c,\\
        -\frac{\sqrt{2}\ell}{2}\frac{D_{\frac{i+i'}{2},\frac{j+j'}{2}}}{|i-i'|\Delta x + |j-j'|\Delta y},& (i',j') \sim (i,j),\,(i,j)\in B_c,\\
        0&\text{others}
    \end{cases}
\end{equation}

\begin{equation}
    \mathbf{s}_{\mathcal{I}(i,j)} =
    \begin{cases}
        S_{i,j},&(i,j)\in I,\\
        0,&(i,j)\in B\cup B_c.
    \end{cases}
\end{equation}

Before discussing further properties of the matrix $\mathbf{L}$, we recall the definition of some special matrices. Given a square matrix $A = (A_{kl})$, its $k$-th row is said to be \textit{weakly diagonally dominant (WDD)} if $|A_{kk}|\geq\sum_{l\neq k}|A_{kl}|$, and the matrix $A$ is said to be WDD if all the rows are WDD. Likewise, its $k$-th row is said to be \textit{strictly diagonally dominant (SDD)} if $\geq$ is replaced by a strict inequality $>$, and the matrix $A$ is said to be SDD if all the rows are WDD.
\begin{defn}
A square matrix $A = (A_{kl})$ is said to be weakly chained diagonally dominant (WCDD) if
\begin{itemize}
    \item $A$ is WDD.
    \item For each row $k$ that is not SDD, 
    %there exists a walk from $k$ to $l$ in the directed graph of $A$ ending at an SDD row $l$, i.e. 
    there exists $k_1,k_2,\dots,k_p$ such that $A_{kk_1},A_{k_1k_2},\dots,A_{k_{p-1}k_p},A_{k_pl}$ are nonzero and the row $A_{l, :}$ is SDD.
\end{itemize}
\end{defn}

\begin{prop}
    $\mathbf{L}$ is a WCDD matrix.
\end{prop}
\begin{proof}
First, we show $\mathbf{L}$ is WDD. As $D>0$, $\sigma_a\geq0$ everywhere, all the off-diagonal terms (see Row 2, 4, 6, 7 in \eqref{eq:Lmatrix}) are non-positive and all the diagonal terms (see Row 1, 3, 5 in \eqref{eq:Lmatrix}) are non-negative. 
It suffices to show that
$$
\mathbf{L}_{\mathcal{I}(i,j),\mathcal{I}(i,j)} \geq \sum_{(i',j') \neq (i,j)} -\mathbf{L}_{\mathcal{I}(i,j),\mathcal{I}(i',j')}.
$$
Move all the terms in this inequality to the left side. It suffices to show that any row sum of $\mathbf{L}$ is non-negative. This is obvious from the definition of $\mathbf{L}$ in~\eqref{eq:Lmatrix}, where the row sum of the $\mathcal{I}(i,j)$-th row is $\sigma_{i,j}$ when $(i,j)\in I$, and the row sum of the $\mathcal{I}(i,j)$-th row is $1$ when $(i,j)\in B\cup B_c$. This proves that $\mathbf{L}$ is WDD. Moreover, the analysis shows that the $\mathcal{I}(i,j)$-th row is SDD when $(i,j)\in B\cup B_c$.

Next, we show the chain condition. If the $\mathcal{I}(i,j)$-th row is not SDD, then $(i,j)\in I$. As the finite difference grid is connected, there exist $(i_1,j_1), \dots, (i_p,j_p)$ such that $(i_p,j_p)\in B\cup B_c$ and $(i,j) \sim (i_1,j_1) \sim \dots \sim (i_p,j_p)$. Notice that the definition of $\mathbf{L}$ has the property that $\mathbf{L}_{\mathcal{I}(i,j),\mathcal{I}(i',j')} <0$ for $(i,j)\sim (i',j')$ (see Row 2,4,6 in~\eqref{eq:Lmatrix}), we conclude the entries $\mathbf{L}_{\mathcal{I}(i,j),\mathcal{I}(i_1,j_1)}$, $\dots$, $\mathbf{L}_{\mathcal{I}(i_{p-1},j_{p-1}),\mathcal{I}(i_p,j_p)}$ are all negative, and the row $\mathbf{L}_{\mathcal{I}(i_p,j_p),:}$ is SDD since $(i_p,j_p)\in B\cup B_c$.
\end{proof}

\begin{prop}[{\cite{shivakumar1974sufficient}}]
WCDD matrices are invertible.
\end{prop}
As a result, the discretized forward problem admits a unique solution $\bm{\phi}_0 = \mathbf{L}^{-1}\mathbf{s}$.

\subsubsection{Discretization of the Adjoint Problem.}
The adjoint problem\eqref{eq:aDE}, \eqref{eq:aDEbc} takes a similar form as the forward problem, except that the source $g$ is imposed on the boundary. Therefore, the adjoint problem can be discretized likewise to yield a linear system
\[\mathbf{L}\bm{\psi} = \mathbf{g}\]
where $\mathbf{L}$ is the same finite difference matrix defined in~\eqref{eq:Lmatrix}, $\bm{\psi}$ consists of the vectorized values of the adjoint solution $\psi$ at black dots such that $\bm{\psi}_{\mathcal{I}(i,j)} = \psi(x_i,y_j)$, and 
\begin{equation}
    \mathbf{g}_{\mathcal{I}(i,j)} =
    \begin{cases}
        0,&(i,j)\in I,\\
        g(x_i,y_j),&(i,j)\in B\cup B_c.
    \end{cases}
\end{equation}

\subsubsection{Discretization of the Internal Data Problem.}
It remains to discretize the internal data problem~\eqref{eq:Hpsi} along with the zero Robin boundary condition. This requires discretizing an operator of the form $D\nabla u\cdot\nabla v = D\nabla v \cdot\nabla u$. The staggered grid scheme gives
\begin{equation}
    \begin{aligned}
        &[D\nabla v\cdot\nabla u]_{i,j} \\
        \approx& \frac{[D\partial_xu\partial_xv]_{i+\frac{1}{2},j} + [D\partial_xu\partial_xv]_{i-\frac{1}{2},j}}{2} + \frac{[D\partial_yu\partial_yv]_{i,j+\frac{1}{2}} + [D\partial_yu\partial_yv]_{i,j-\frac{1}{2}}}{2} \\
        \approx& \frac{[D\partial_xv]_{i+\frac{1}{2},j}[u_{i+1,j}-u_{i,j}] + [D\partial_xv]_{i-\frac{1}{2},j}[u_{i,j}-u_{i-1,j}]}{2\Delta x}\\
         &+ \frac{[D\partial_yv]_{i,j+\frac{1}{2}}[u_{i,j+1}-u_{i,j}] + [D\partial_yv]_{i,j-\frac{1}{2}}[u_{i,j}-u_{i,j-1}]}{2\Delta y} \\
        =&\left[\frac{D_{i+\frac{1}{2},j}[v_{i+1,j}-v_{i,j}]}{2\Delta x^2}\right]u_{i+1,j} + \left[\frac{D_{i-\frac{1}{2},j}[v_{i-1,j}-v_{i,j}]}{2\Delta x^2}\right]u_{i-1,j}\\
         &+ \left[\frac{D_{i,j+\frac{1}{2}}[v_{i,j+1}-v_{i,j}]}{2\Delta y^2}\right]u_{i,j+1} + \left[\frac{D_{i,j-\frac{1}{2}}[v_{i,j-1}-v_{i,j}]}{2\Delta y^2}\right]u_{i,j-1}\\
         &-\left[\frac{D_{i+\frac{1}{2},j}[v_{i+1,j}-v_{i,j}]}{2\Delta x^2}+\frac{D_{i-\frac{1}{2},j}[v_{i-1,j}-v_{i,j}]}{2\Delta x^2}+\frac{D_{i,j+\frac{1}{2}}[v_{i,j+1}-v_{i,j}]}{2\Delta y^2}+\frac{D_{i,j-\frac{1}{2}}[v_{i,j-1}-v_{i,j}]}{2\Delta y^2}\right]u_{i,j},
    \end{aligned}
\end{equation}
The discretization of~\eqref{eq:Hpsi} becomes
\[\mathbf{A}_{\bm{\psi_0}}\bm{\phi_0} = \mathbf{h}_{\bm{\psi_0}}\]
where $\bm{\phi_0}$ consists of the vectorized values of the forward solution $\phi_0$ at black dots such that $\bm{\phi}_{0\mathcal{I}(i,j)} = \phi(x_i,y_j)$, and
\begin{equation}
    (\mathbf{A}_{\psi})_{\mathcal{I}(i,j),\mathcal{I}(i',j')} = 
    \begin{cases}
        -\sum_{(\tilde{i},\tilde{j})\sim(i,j)}\frac{D_{\frac{i+\tilde{i}}{2},\frac{j+\tilde{j}}{2}}[\psi_{i,j}+\frac{2\gamma-1}{2}[\psi_{\tilde{i},\tilde{j}}-\psi_{i,j}]]}{|i-\tilde{i}|\Delta x^2 + |j-\tilde{j}|\Delta y^2} + 2\gamma\sigma_{i,j}\psi_{i,j},&(i',j') = (i,j),\,(i,j)\in I\\
        \frac{D_{\frac{i+i'}{2},\frac{j+j'}{2}}[\psi_{i,j}+\frac{2\gamma-1}{2}[\psi_{i',j'}-\psi_{i,j}]]}{|i-i'|\Delta x^2 + |j-j'|\Delta y^2},&(i',j') \sim (i,j),\,(i,j)\in I\\
        1+\ell\sum_{I\ni(\tilde{i},\tilde{j})\sim(i,j)}\frac{D_{\frac{i+\tilde{i}}{2},\frac{j+\tilde{j}}{2}}}{|i-\tilde{i}|\Delta x + |j-\tilde{j}|\Delta y},& (i',j') = (i,j),\,(i,j)\in B\\
        -\ell\frac{D_{\frac{i+i'}{2},\frac{j+j'}{2}}}{|i-i'|\Delta x + |j-j'|\Delta y},& I\ni(i',j') \sim (i,j)\in B,\\
        1+\frac{\sqrt{2}\ell}{2}\sum_{(\tilde{i},\tilde{j})\sim(i,j)}\frac{D_{\frac{i+\tilde{i}}{2},\frac{j+\tilde{j}}{2}}}{|i-\tilde{i}|\Delta x + |j-\tilde{j}|\Delta y},& (i',j') = (i,j),\,(i,j)\in B_c,\\
        -\frac{\sqrt{2}\ell}{2}\frac{D_{\frac{i+i'}{2},\frac{j+j'}{2}}}{|i-i'|\Delta x + |j-j'|\Delta y},& (i',j') \sim (i,j),\,(i,j)\in B_c,\\
        0&\text{others}        
    \end{cases}
\end{equation}

\begin{equation}
    (\mathbf{h}_{\psi})_{\mathcal{I}(i,j)} =
    \begin{cases}
        (H_\psi)_{i,j},&(i,j)\in I,\\
        0,&(i,j)\in B\cup B_c.
    \end{cases}
\end{equation}

\subsubsection{Discrete Uncertainty Quantification Estimate.}
Parallel to Theorem~\ref{thm:continuousUQ}, we derive the following UQ estimate for the discretized model. Note that the uncertainties of the optical parameters $(D,\sigma_a)$ are implicitly encoded in the difference $\tilde{\mathbf{L}}-\mathbf{L}$ and $\tilde{\mathbf{A}}_{\tilde{\bm{\phi_0}}}-\mathbf{A}_{\bm{\phi_0}}$.

\begin{thm} \label{thm:discreteUQ}
Suppose $0$ is not an eigenvalue of $\mathbf{A}_{\bm{\psi_0}}$ and $\tilde{\mathbf{A}}_{\bm{\tilde{\psi}_0}}$ for some $\bm{\psi_0}>0$ and $\bm{\tilde{\psi}_0} > 0$, then 
\begin{equation}
    \begin{aligned}
    \|\tilde{\mathbf{s}}-\mathbf{s}\|_2\leq&\|\mathbf{h}_{\bm{\phi_0}}\|_2(\|\mathbf{A}_{\psi_0}^{-1}\|_2\|\tilde{\mathbf{L}}-\mathbf{L}\|_2+\|\tilde{\mathbf{L}}\|_2\|\tilde{\mathbf{A}}_{\tilde{\bm{\phi_0}}}^{-1}\|_2\|\mathbf{A}_{\bm{\phi_0}}^{-1}\|_2\|\tilde{\mathbf{A}}_{\tilde{\bm{\phi_0}}}-\mathbf{A}_{\bm{\phi_0}}\|_2).
    \end{aligned}
\end{equation}
\end{thm}

\begin{proof}
Under the assumption, the matrix $\mathbf{A}_{\bm{\psi_0}}$ is invertible for some $\bm{\psi_0}>0$. We can represent $\bm{\phi_0} = \mathbf{A}_{\bm{\psi_0}}^{-1}\mathbf{h}_{\psi_0}$, then $\mathbf{s}=\mathbf{L}\bm{\phi_0} = \mathbf{L}\mathbf{A}_{\bm{\psi_0}}^{-1}\mathbf{h}_{\bm{\psi_0}}$. Therefore,
\begin{equation}
    \begin{aligned}
    \|\tilde{\mathbf{s}}-\mathbf{s}\|_2=&\|(\tilde{\mathbf{L}}\tilde{\mathbf{A}}_{\tilde{\bm{\phi_0}}}^{-1}-\mathbf{L}\mathbf{A}_{\bm{\phi_0}}^{-1})\mathbf{h}_{\bm{\phi_0}}\|_2\\
    \leq&\|\tilde{\mathbf{L}}\tilde{\mathbf{A}}_{\tilde{\bm{\phi_0}}}^{-1}-\mathbf{L}\mathbf{A}_{\bm{\phi_0}}^{-1}\|_2\|\mathbf{h}_{\bm{\phi_0}}\|_2\\
    \leq&(\|(\tilde{\mathbf{L}}-\mathbf{L})\mathbf{A}_{\bm{\phi_0}}^{-1}\|_2+\|\tilde{\mathbf{L}}(\tilde{\mathbf{A}}_{\tilde{\bm{\phi_0}}}^{-1}-\mathbf{A}_{\bm{\phi_0}}^{-1})\|_2)\|\mathbf{h}_{\bm{\phi_0}}\|_2\\
    \leq&(\|\tilde{\mathbf{L}}-\mathbf{L}\|_2\|\mathbf{A}_{\bm{\phi_0}}^{-1}\|_2+\|\tilde{\mathbf{L}}\|_2\|\tilde{\mathbf{A}}_{\tilde{\bm{\phi_0}}}^{-1}-\mathbf{A}_{\bm{\phi_0}}^{-1}\|_2)\|\mathbf{h}_{\bm{\phi_0}}\|_2
    \end{aligned}
\end{equation}
where $\|\cdot\|_2$ denotes the vector/matrix 2-norm.
Using the relation $A^{-1}-B^{-1}=A^{-1}(B-A)B^{-1}$, we obtain the desired estimate.
\end{proof}

\subsection{Discretization with Anisotropic Diffusion Coefficients.}
When $D$ is anisotropic, i.e, a symmetric positive definition matrix-valued function, the operators $\nabla\cdot D\nabla$ and $D\nabla v\cdot\nabla$ can be discretized as follows
\begin{align*}
    [\nabla\cdot D\nabla u]_{i,j} &= \frac{[(D\nabla u)_1]_{i+\frac{1}{2},j} - [(D\nabla u)_1]_{i-\frac{1}{2},j}}{\Delta x} + \frac{[(D\nabla u)_2]_{i,j+\frac{1}{2}}-[(D\nabla u)_2]_{i,j-\frac{1}{2}}}{\Delta y}\\
    [D\nabla v\cdot\nabla u]_{i,j} &= \frac{[(D\nabla v)_1\partial_xu]_{i+\frac{1}{2},j} + [(D\nabla v)_2\partial_xu]_{i-\frac{1}{2},j}}{2} + \frac{[(D\nabla v)_2\partial_yu]_{i,j+\frac{1}{2}} + [(D\nabla v)_2\partial_yu]_{i,j-\frac{1}{2}}}{2}   
\end{align*}
where $(D\nabla u)_1$ (resp. $(D\nabla u)_2$) denotes the first (resp. second) component of the vector $D\nabla u$.
The discretization now differs from the isotropic case. This is because for an isotropic $D$
$$
(D\nabla u)_1 = D \partial_x u, \qquad
(D\nabla u)_2 = D \partial_y u
$$
which only requires $[\partial_x u]_{i+\frac{1}{2},j}$ and $[\partial_yu]_{i,j+\frac{1}{2}}$ in the staggered grid. However, for an anisotropic $D$:
$$
(D\nabla u)_1 = D_{11} \partial_x u + D_{12} \partial_y u, \qquad
(D\nabla u)_2 = D_{21} \partial_x u + D_{22} \partial_y u
$$
which requires two additional terms $[\partial_xu]_{i,j+\frac{1}{2}}$ and $[\partial_yu]_{i+\frac{1}{2},j}$. These additional terms can be discretized as follows:
\begin{align*}
    [\partial_yu]_{i+\frac{1}{2},j} &= \frac{[\partial_yu]_{i,j} + [\partial_yu]_{i+1,j}}{2} = \frac{u_{i+1,j+1}+u_{i,j+1}-u_{i,j-1}-u_{i+1,j-1}}{4\Delta y},\\
    [\partial_xu]_{i,j+\frac{1}{2}} &= \frac{[\partial_xu]_{i,j} + [\partial_xu]_{i,j+1}}{2} = \frac{u_{i+1,j+1}+u_{i+1,j}-u_{i-1,j}-u_{i-1,j+1}}{4\Delta x},    
\end{align*}
see~\cite{gunter2005modelling} for the detail. This discretization results in a matrix $\mathbf{L}$. The rest of the analysis is similar provided $\mathbf{L}$ is invertible, and we can obtain Theorem~\ref{thm:discreteUQ} as well.

\bigskip
\section{Numerical Experiment} \label{sec:numerics}

In this section, we demonstrate numerical experiments to validate the reconstruction procedure and quanfity the impact of inaccurate optical coefficients $(D,\sigma_a)$ to the source recovery. We will restrict the discussion in this section to isotropic $D$ for the ease of notations.

\subsection{Uncertainty Generation}
We will utilize the generalized Polynomial Chaos Expansion (PCE) to facilitate generation of uncertainty. PCE approximates a well-behaved random variable using a series of polynomials under certain probability distribution. Specifically, let $(X,\mathcal{F},\mathbb{P})$ be a probability space, and let $\xi(\omega)$ be a random variable (where $\omega \in X$ is a sample) with probability density function $p(t)$. 
Suppose a deterministic ground-truth $\mathfrak{u} = \mathfrak{u}(x)$ is given, then the uncertainty generated by PCE takes the form 
\begin{equation} \label{eq:PCE}
\mathfrak{u}(x,\xi(\omega))=\sum^\infty_{k=0}\mathfrak{u}_k(x)\Phi_k(\xi(\omega)), \qquad (x,\omega)\in \Omega\times X
\end{equation}
where $\mathfrak{u}_k(x)$'s are the coefficients, $\mathfrak{u}_0$ is the ground truth, $\Phi_0=1$, $\Phi_k$'s are orthogonal polynomials, that is,
\[\int_\mathbb{R} \Phi_i(t)\Phi_j(t)p(t)\dif t=\delta_{ij}.\]
For the numerical experiments, $\xi$ is chosen to be uniformly distributed on the sample space $X=[-1,1]$; $\Phi_k$'s are the Legendre polynomials on $[-1,1]$; the PCE is truncated at $k=K_c$. Then
\[ \mathbb{E}[\mathfrak{u}]=\mathfrak{u}_0,\qquad\textup{Var}[\mathfrak{u}]=\sum_{k=1}^{K_c}\mathfrak{u}_k^2.\]

In the subsequent numerical experiments, we inject uncertainties into the optical coefficients $(D,\sigma_a)$ based on the following process:
\begin{enumerate}[(1)]
    \item Generate the coefficients ${\mathfrak{u}_D}_k, {\mathfrak{u}_{\sigma_a}}_k$ using the truncated Fourier series in $x$:
$$
    \begin{aligned}
        {\mathfrak{u}_D}_k &= \sum_{\|\mathbf{n}\|_\infty=k}c_{1\mathbf{n}}\sin(\pi\mathbf{n}\cdot x)+c_{2\mathbf{n}}\cos(\pi\mathbf{n}\cdot x),\\
        {\mathfrak{u}_{\sigma_a}}_k &= \sum_{\|\mathbf{n}\|_\infty=k}c_{3\mathbf{n}}\sin(\pi\mathbf{n}\cdot x)+c_{4\mathbf{n}}\cos(\pi\mathbf{n}\cdot x).
    \end{aligned}    
$$
Here $\mathbf{n}\in\mathbb{Z}^n$, the Fourier coefficients $c_{1\mathbf{n}},c_{2\mathbf{n}},c_{3\mathbf{n}},c_{4\mathbf{n}}$ are independently chosen from the uniform distributions on $[-1,1]$. Once generated, they are fixed so that the coefficients ${\mathfrak{u}_D}_k, {\mathfrak{u}_{\sigma_a}}_k$ are deterministic.
    \item Randomly generate $\xi$ from the uniform distribution on $[-1,1]$, then construct the uncertainties ${\mathfrak{u}_D},\mathfrak{u}_{\sigma_a}$ according to~\eqref{eq:PCE} with $k=1,2,\dots,10$:
    \begin{align*}
        {\mathfrak{u}_D} &\coloneqq  \sum_{k=1}^{10}{\mathfrak{u}_D}_k\Phi_k(\xi(\omega)),\qquad{\mathfrak{u}_{\sigma_a}} \coloneqq \sum_{k=1}^{10}{\mathfrak{u}_{\sigma_a}}_k\Phi_k(\xi(\omega))
    \end{align*}
Note that $\mathbb{E}[{\mathfrak{u}_D}] = \mathbb{E}[\mathfrak{u}_{\sigma_a}] = 0$.
    \item Once the uncertainties are generated, we rescale the uncertainties based on prescribed relative uncertainty levels $e_D, e_{\sigma_a}$ to construct the optical coefficients with uncertainty $(\tilde{D},\tilde{\sigma}_a)$ as follows:
\begin{equation} \label{eq:Dsigmatilde}
    \begin{aligned}
    \tilde{D} & := D + \frac{\mathfrak{u}_De_D}{\|\mathfrak{u}_D\|_{H^1}}\|D\|_{H^1},\\
    \tilde{\sigma}_a & := \sigma_a + \frac{\mathfrak{u}_{\sigma_a}e_{\sigma_a}}{\|\mathfrak{u}_{\sigma_a}\|_{L^2}}\|\sigma_a\|_{L^2}.
\end{aligned}
\end{equation}
\end{enumerate}

The impact of the inaccuracy in the optical coefficients will be quantitatively measured by the relative standard deviation defined as follows:
\begin{equation} \label{eq:relstd}
\mathcal{E}_S\coloneqq\frac{\sqrt{\mathbb{E}[\|\tilde{S}-S\|^2_{L^2}]}}{\|S\|_{L^2}},\quad\mathcal{E}_D\coloneqq\frac{\sqrt{\mathbb{E}[\|\tilde{D}-D\|^2_{H^1}]}}{\|D\|_{H^1}},\quad\mathcal{E}_{\sigma_a}\coloneqq\frac{\sqrt{\mathbb{E}[\|\tilde{\sigma}_a-\sigma_a\|_{L^2}^2]}}{\|\sigma_a\|_{L^2}}.
\end{equation}
Note that $\mathcal{E}_D=e_D$ and $\mathcal{E}_{\sigma_a} = e_{\sigma_a}$ are precisely the relative uncertainty levels that are used to define $(\tilde{D},\tilde{\sigma}_a)$ in~\eqref{eq:Dsigmatilde}. This justifies that the relative standard deviation is a reasonable quantity to measure the uncertainty. In the following, we will specify various uncertainty levels $e_D, e_{\sigma_a}$ and plot $\mathcal{E}_S$ versus them, see Figure~\ref{fig:exp1_stability} and Figure~\ref{fig:exp2_stability}.

\subsection{Numerical Implementation.}
We choose the 2D computational domain $\Omega=[-1,1]^2$. The diffusion equation is solved using the staggered grid scheme outlined in Section~\ref{sec:discrete}. To avoid the inverse crime, the forward problem is solved on a fine mesh with step size $h=\frac{1}{200}$, while the inverse problem is solved on a coarse mesh with step size $h=\frac{1}{100}$ using re-sampled data. We numerically calculate the noise-free $\phi_0$ and $\psi_0$ using ground truth $S$ and $(D,\sigma_a)$, here we choose $\psi_0>0$ by solving \eqref{eq:aDE} with a positive Dirichlet boundary condition. The resulting Robin boundary condition is the corresponding $g$ in~\eqref{eq:aDEbc}. Once we have $\phi_0$ and $\psi_0$, we can calculate the internal data $H_{\psi_0}$ through~\eqref{eq:Hpsi}. 
Note that the internal data is derived from the boundary measurement, hence is independent of the uncertainty on the optical coefficients.

\textbf{Experiment 1.} In this experiment, we consider the case that the optical coefficients can be represented using low-frequency Fourier basis. We choose
\[D=\cos^2(x+2y)-3\sin^2(3x-4y)+5,\qquad\sigma_a=\cos^2(5x)+\sin^2(5y)+1,\]
and the source $S$ to be the Shepp-Logan phantom, see Figure~\ref{fig:exp1_coefficients}.

Using the ground-truth $(D,\sigma_a)$, we generate the uncertainties according to~\eqref{eq:Dsigmatilde} to obtain 1000 samples of inaccurate optical coefficients $(\tilde{D},\tilde{\sigma}_a)$. Set $\Delta D := \tilde{D} - D$ and $\Delta\sigma_a = \tilde{\sigma}_a - \sigma_a$.
We implemented the reconstruction procedure 1000 times to plot the distribution of $\|\Delta S\|_{L^2}$ versus $\|\Delta D\|_{H^1}$ and $\|\Delta \sigma_a\|_{L^2}$, see Figure~\ref{fig:exp1_distribution}. 
It is clear that for fixed $\|\Delta D\|_{H^1}$, $\|\Delta S\|_{L^2}$ is more concentrated compared to fixed $\|\Delta \sigma_a\|_{L^2}$, suggesting that the uncertainty in $\tilde{D}$ has larger impact to the reconstruction than the uncertainty in $\tilde{\sigma}_a$. Moreover, the distribution of the scatter plot suggests that $\|\Delta S\|_{L^2}$ is locally Lipschitz stable with respect to $\|\Delta D\|_{H^1}$ for small $\Delta D$, agreeing with the estimates in Theorem~\ref{thm:continuousUQ} and Theorem~\ref{thm:discreteUQ}.
One of the reconstructions is illustrated in Figure~\ref{fig:exp1_reconstruction}, and the average of the $1000$ reconstructed sources is illustrated in Figure~\ref{fig:exp1_average}. We see that the averaged $\tilde{S}$ is close to the ground truth $S$. This can be understood as follows. Let us view $S=\mathcal{S}[D,\sigma_a]$ as a nonlinear functional of $(D,\sigma_a)$. When small perturbations $(\delta D, \delta\sigma_a)$ are added, the response perturbation $\delta S \approx d\mathcal{S}(\delta D, \delta\sigma_a)$ depends almost linearly on $(\delta D, \delta\sigma_a)$ where $d\mathcal{S}$ is the Frech\'et derivative. Hence $\mathbb{E}[S] \approx d\mathcal{S}(\mathbb{E}[\delta D], \mathbb{E}[\delta\sigma_a]) = 0$.

To better understand the relations between $\mathcal{E}_S$ versus $\mathcal{E}_D$ (resp. $\mathcal{E}_S$ versus $\mathcal{E}_{\sigma_a}$), we take $\Delta\sigma_a=0$ (resp. $\Delta D=0$) and add $e_{D} = 2\%,4\%,6\%,8\%,10\%$ of random noise to $D$ (resp. $e_{\sigma_a} = 2\%,4\%,6\%,8\%,10\%$ of random noise to $\sigma_a$). The plots are shown in Figure~\ref{fig:exp1_stability}.
We observe that $\mathcal{E}_S$ depends linearly or superlinearly on $\mathcal{E}_D$ and $\mathcal{E}_{\sigma_a}$, and the same level of relative uncertainty on $D$ has larger impact than on $\sigma_a$.
Note that the plotted curves are nonlinear because the constant factors $C_{1ij}, C_2$ in Theorem~\ref{thm:continuousUQ} also depend on $(\tilde{D},\tilde{\sigma}_a)$.

\begin{remark}\label{remark:example}
If $X$ is a random variable and $f$ is a nonlinear function, it is generally not true that $\mathbb{E}f(X)\neq f(\mathbb{E}(X))$. For example, if we choose a uniformly distributed random variable $X\sim U\left(-1,1\right)$ and a nonlinear function $f_\alpha(x) := |x|^\alpha$ ($0<\alpha<1$). Then $\mathbb{E}[X] = 0$, hence $f_\alpha(\mathbb{E}[X]) = f_\alpha(0) = 0$. But 
$$
0< \mathbb{E}[f_\alpha(X)] = \frac{1}{2} \int^{1}_{-1} |x|^\alpha \,\dif x = \frac{1}{\alpha+1} < 1
$$
and $\mathbb{E}[f_\alpha(X)]$ monotonically increases to $1$ as $\alpha\to 0_+$.
\end{remark}

\begin{figure}[h]
    \centering
    \includegraphics[width=0.3\textwidth]{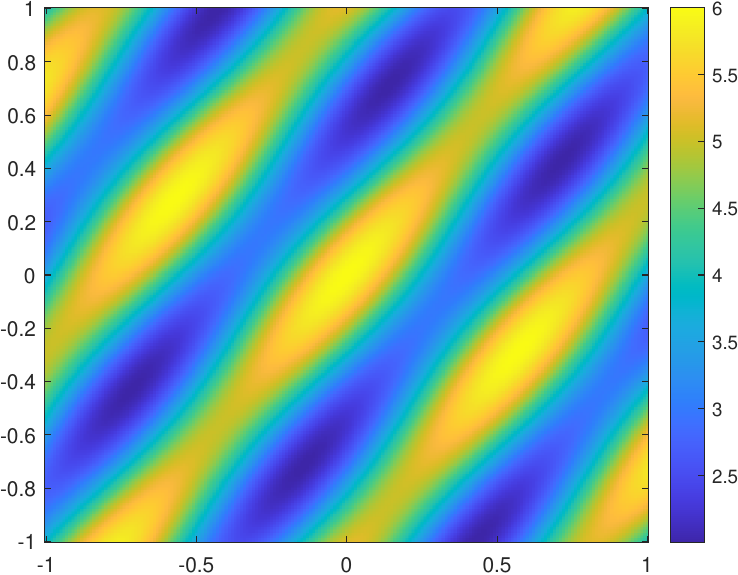}
    \includegraphics[width=0.3\textwidth]{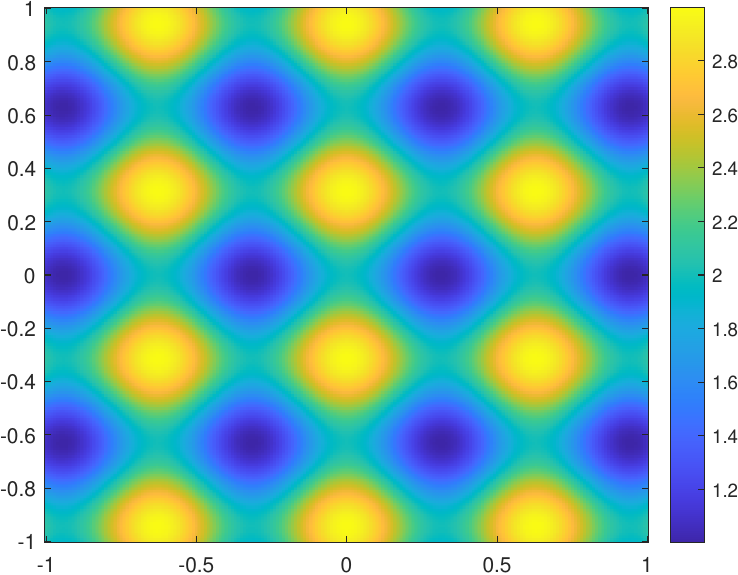}
    \includegraphics[width=0.3\textwidth]{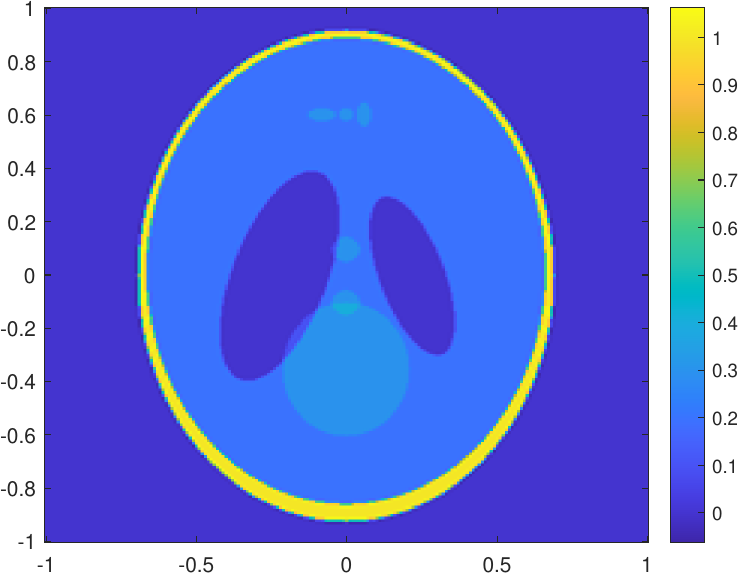}
    \caption{Left: Diffusion coefficient $D$. Middle: Absorption coefficient $\sigma_a$. Right: Shepp-Logan Source $S$.}
    \label{fig:exp1_coefficients}
\end{figure}

\begin{figure}[h]
\centering
\begin{subfigure}[b]{.6\textwidth}
\includegraphics[width=\textwidth]{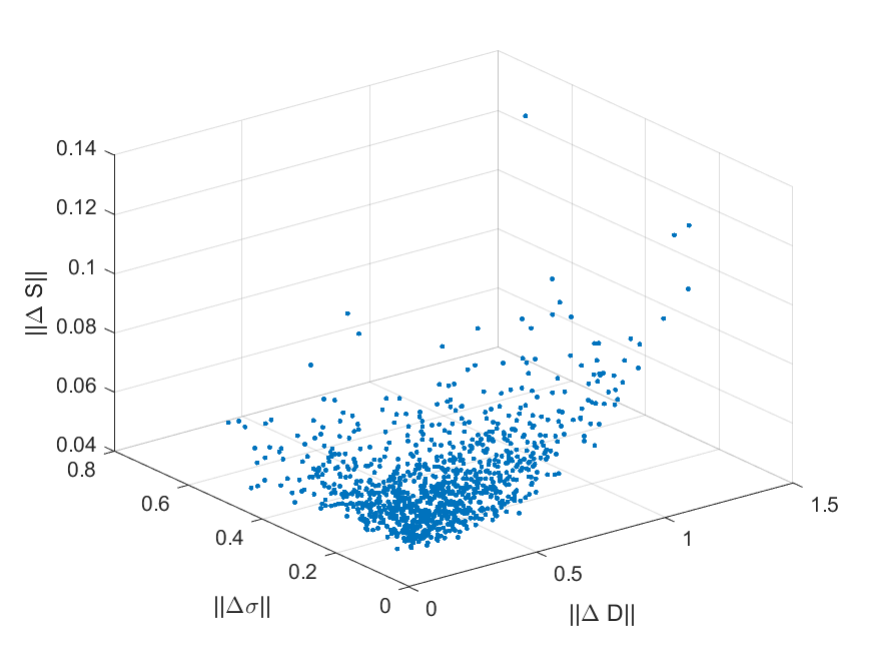}
\end{subfigure}\qquad
\begin{subfigure}[b]{.3\textwidth}
\includegraphics[width=\textwidth]{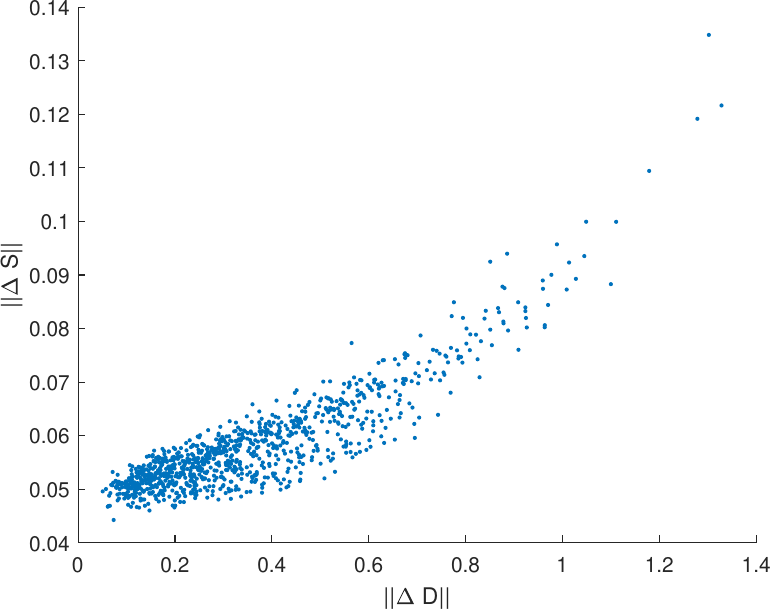}

\vspace{2ex}

\includegraphics[width=\textwidth]{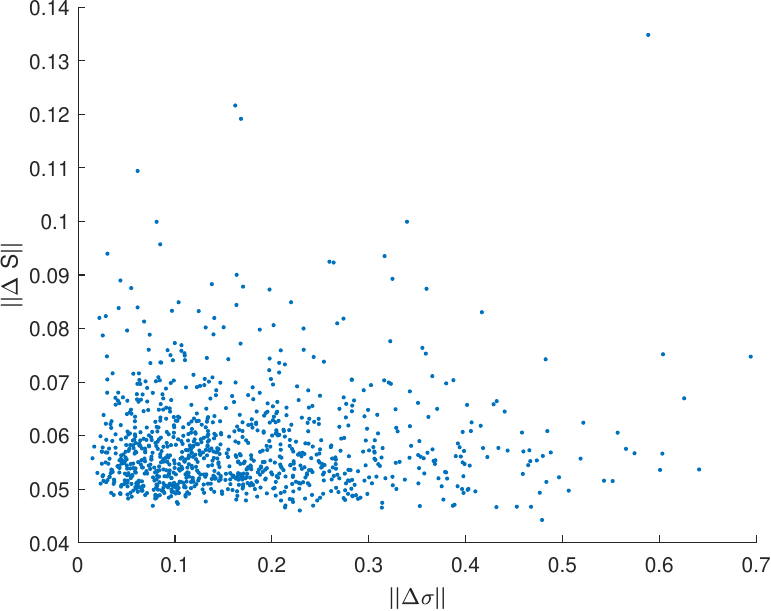}
\end{subfigure}
\caption{The distribution of the error with respect to the inaccuracies in optical coefficients}
\label{fig:exp1_distribution}
\end{figure}

\begin{figure}[h]
    \centering
    \includegraphics[width=0.4\textwidth]{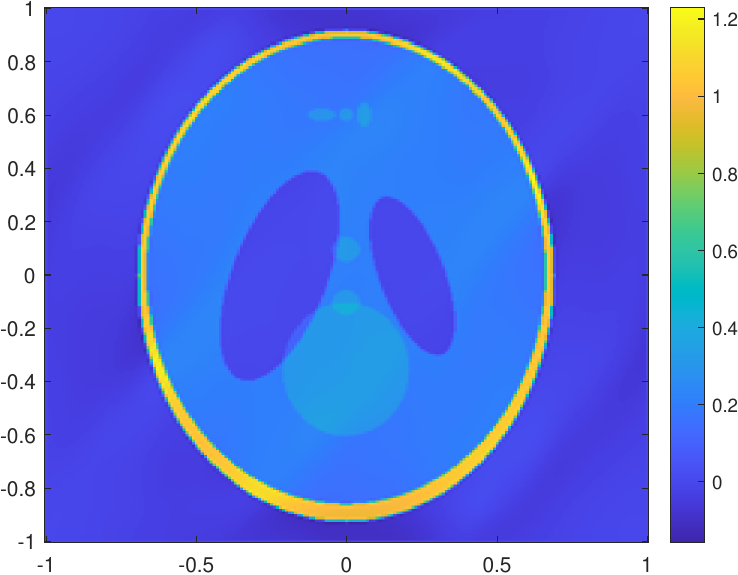}
    \includegraphics[width=0.4\textwidth]{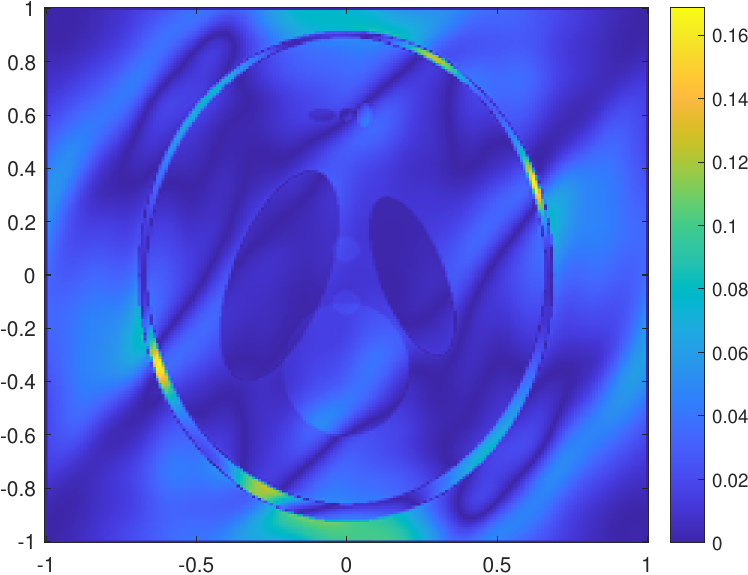}
    \caption{Reconstructed source $\tilde{S}$ and its error under $10\%$ Gaussian random noise.}
    \label{fig:exp1_reconstruction}
    \includegraphics[width=0.4\textwidth]{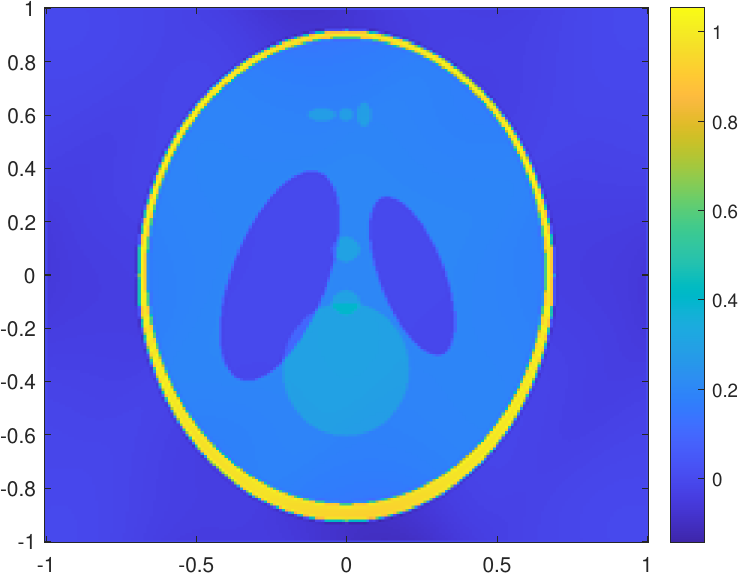}
    \includegraphics[width=0.4\textwidth]{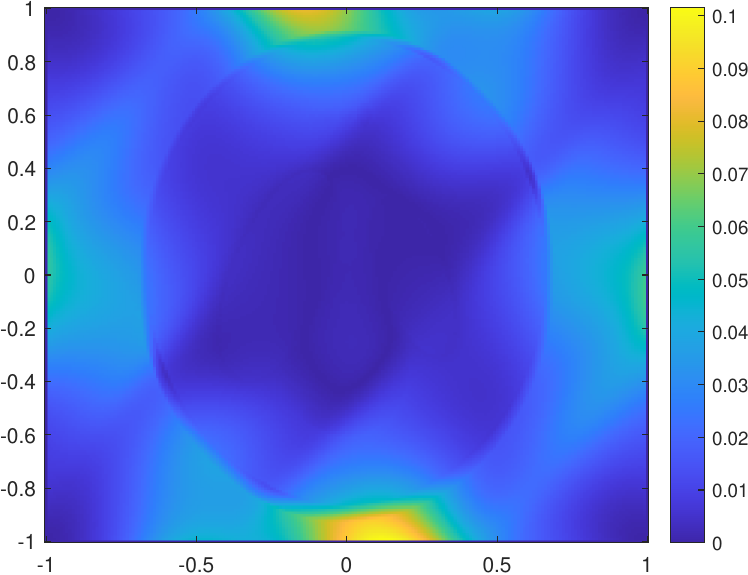}
    \caption{Averaged reconstructed source $\tilde{S}$ and its error under $10\%$ Gaussian random noise.}
    \label{fig:exp1_average}
    \includegraphics[width=0.4\textwidth]{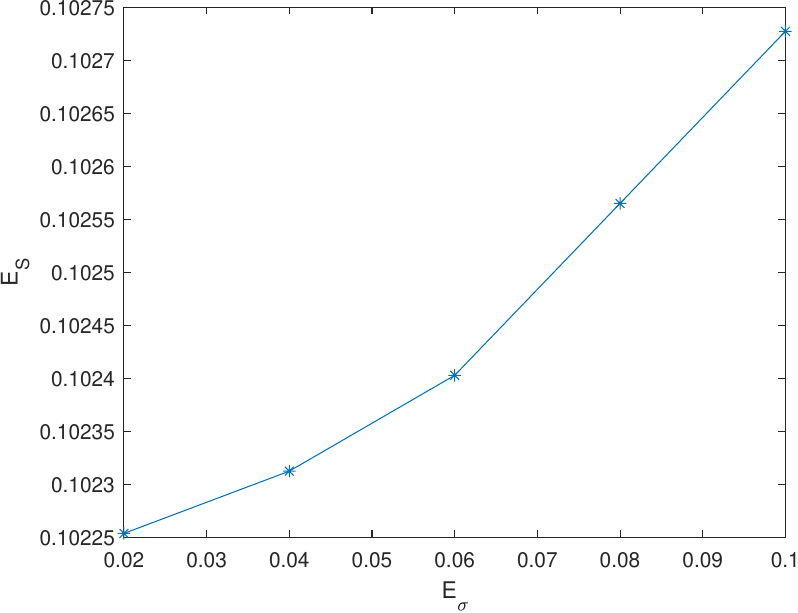}
    \includegraphics[width=0.4\textwidth]{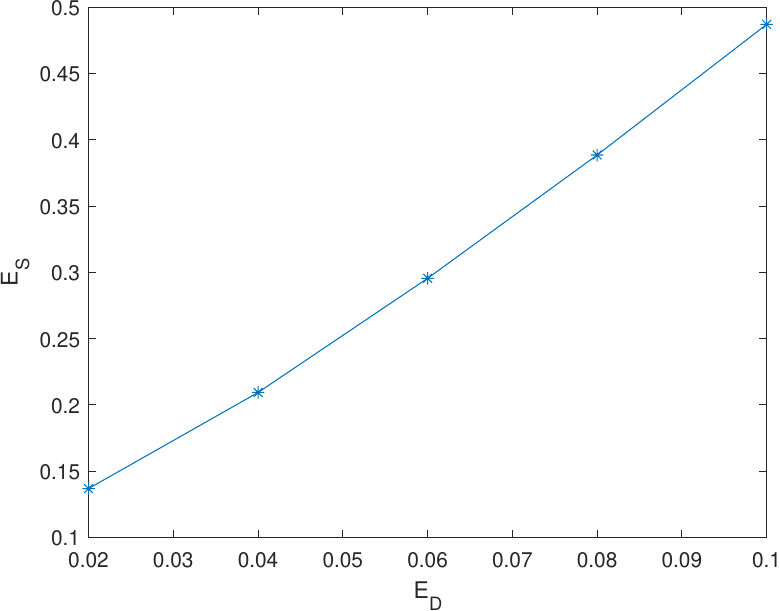}
    \caption{Left: $\mathcal{E}_S$ versus $\mathcal{E}_{\sigma_a}$. Right: $\mathcal{E}_S$ versus $\mathcal{E}_{D}$. }
    \label{fig:exp1_stability}
\end{figure}

\subsubsection{Experiment 2:} In this experiment, we consider the case that the optical coefficients can not be represented using the low frequency Fourier basis. We choose
\[D=3-\max\{|x|,|y|\},\qquad\sigma_a=\frac{3}{2}-\frac{1}{2}\text{sgn}\left(x^2+y^2-\frac{4}{5}\right),\]
and we choose the source $S$ to be the Shepp-Logan phantom, see Figure~\ref{fig:exp2_coefficients}. We choose the relative uncertainty level at $10\%$ and run $1000$ reconstructions to plot the distribution of $\|\tilde{S}-S\|_{L^2}$ versus $\|\tilde{D}-D\|_{H^1}$ and $\|\tilde{\sigma}_a-\sigma_a\|_{L^2}$, see Figure~\ref{fig:exp2_distribution}. One of the reconstructions is illustrated in Figure~\ref{fig:exp2_reconstruction}, and the average of $1000$ reconstructed sources is illustrated in Figure~\ref{fig:exp2_average}. For the relation between the relative standard deviations, we fix $D$ and $\sigma_a$ respectively and add $2\%,4\%,6\%,8\%,10\%$ Gaussian random noise to another optical coefficient. 
The relations are shown in Figure~\ref{fig:exp2_stability}.
Again, we observe that uncertainties in $D$ have larger impact to the reconstruction than that in $\sigma_a$. We also observe that the averaging process reduces the uncertainty in the reconstruction. 
The impact $\mathcal{E}_S$ also depends linearly or superlinearly on $\mathcal{E}_D$ and $\mathcal{E}_{\sigma_a}$.

\begin{figure}[h]
    \centering
    \includegraphics[width=0.3\textwidth]{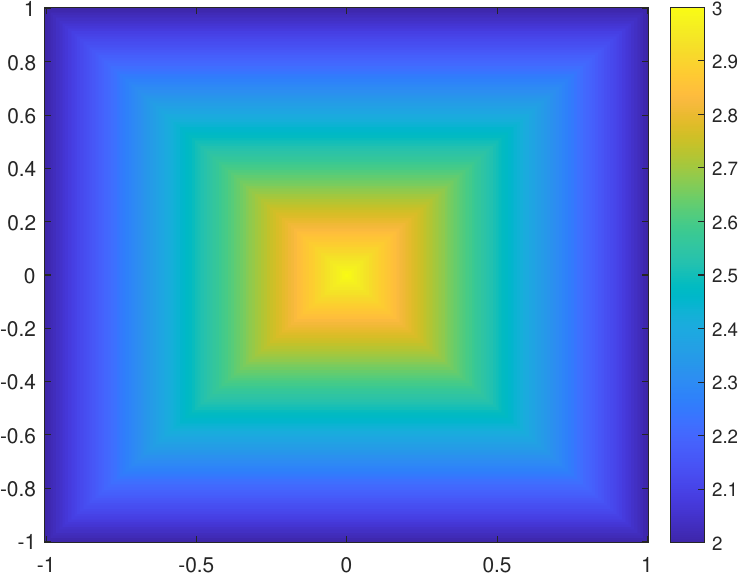}
    \includegraphics[width=0.3\textwidth]{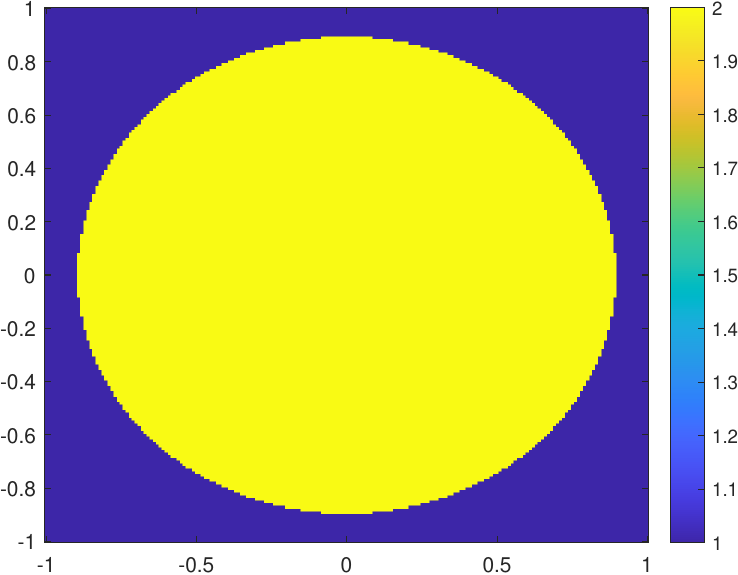}
    \includegraphics[width=0.3\textwidth]{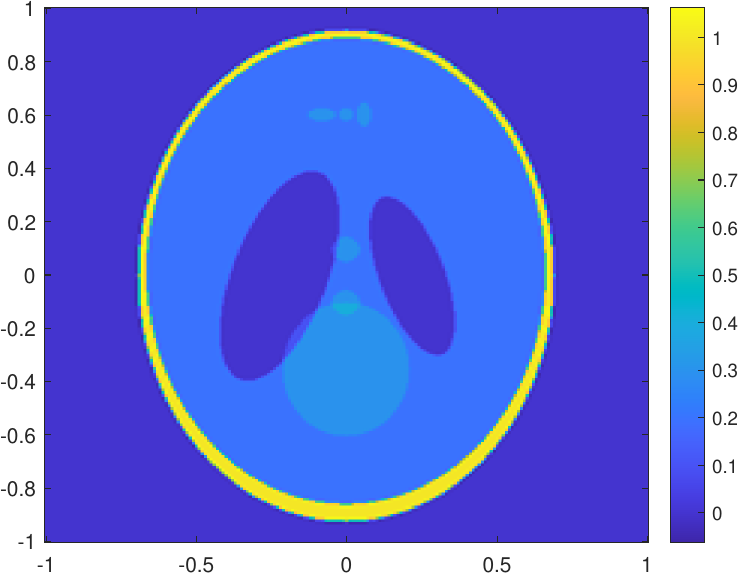}
    \caption{Left: Diffusion coefficient $D$. Middle: Absorption coefficient $\sigma_a$. Right: Shepp-Logan Source}
    \label{fig:exp2_coefficients}
\end{figure}

\begin{figure}[h]
\centering
\begin{subfigure}[b]{.6\textwidth}
\includegraphics[width=\textwidth]{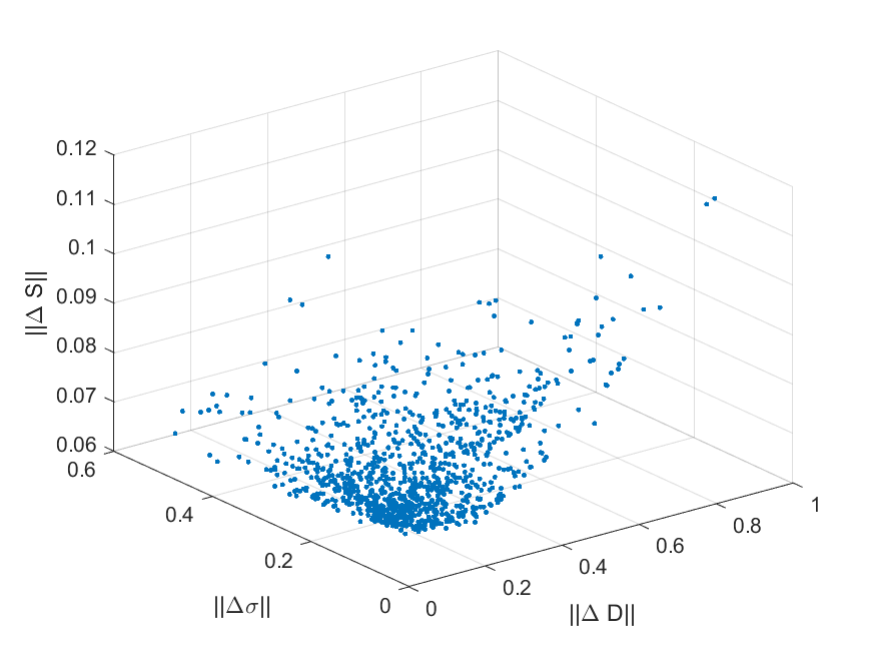}
\end{subfigure}\qquad
\begin{subfigure}[b]{.3\textwidth}
\includegraphics[width=\textwidth]{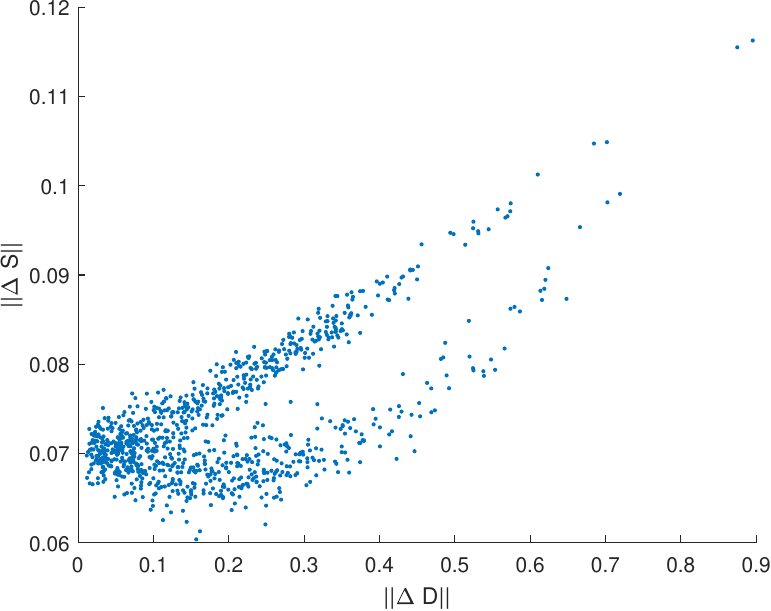}

\vspace{2ex}

\includegraphics[width=\textwidth]{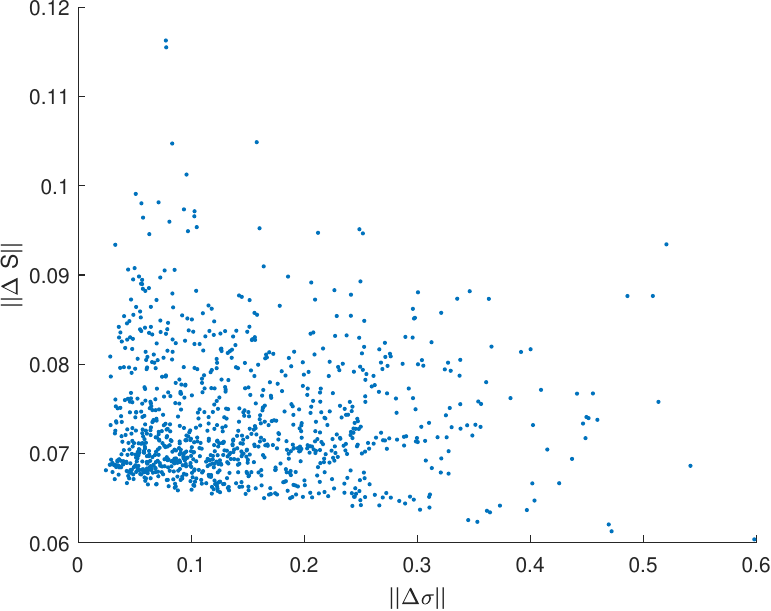}
\end{subfigure}
\caption{The distribution of the error $\|\Delta S\|$ with respect to the inaccuracies $\|\Delta D\|$ and $\|\Delta \sigma_a\|$.}
\label{fig:exp2_distribution}
\end{figure}

\begin{figure}[h]
    \centering
    \includegraphics[width=0.4\textwidth]{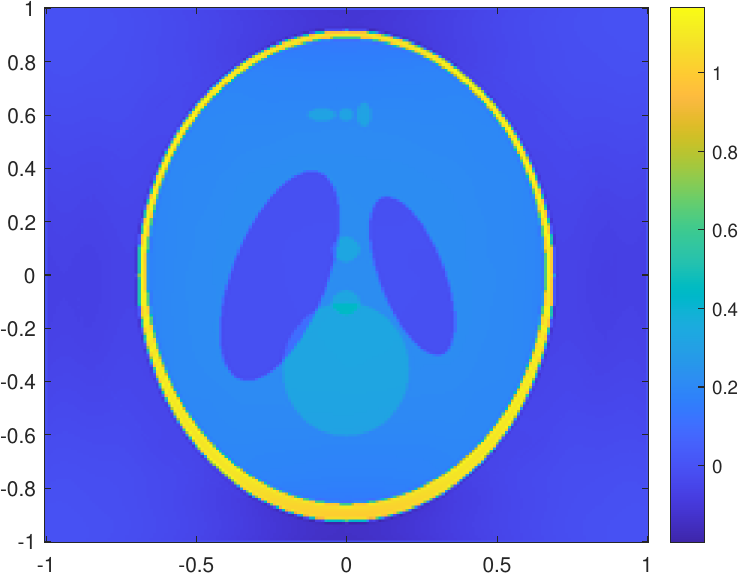}
    \includegraphics[width=0.4\textwidth]{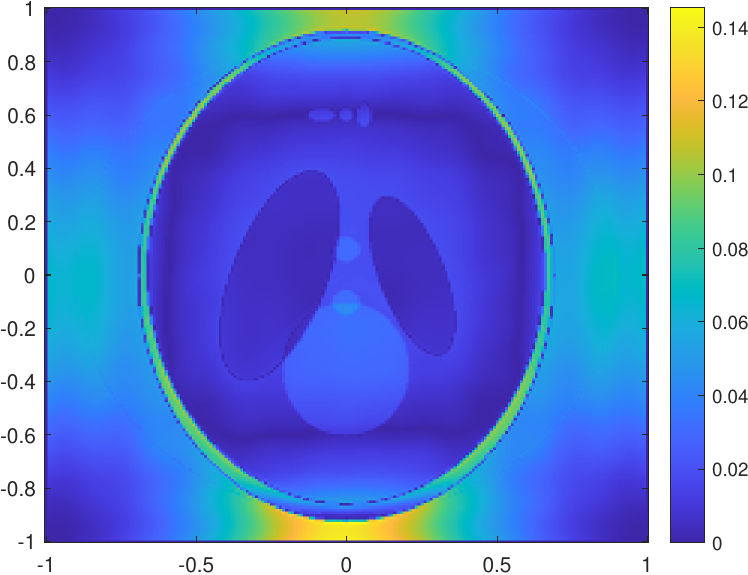}
    \caption{Reconstructed source $\tilde{S}$ and its error under $10\%$ Gaussian random noise.}
    \label{fig:exp2_reconstruction}

    \includegraphics[width=0.4\textwidth]{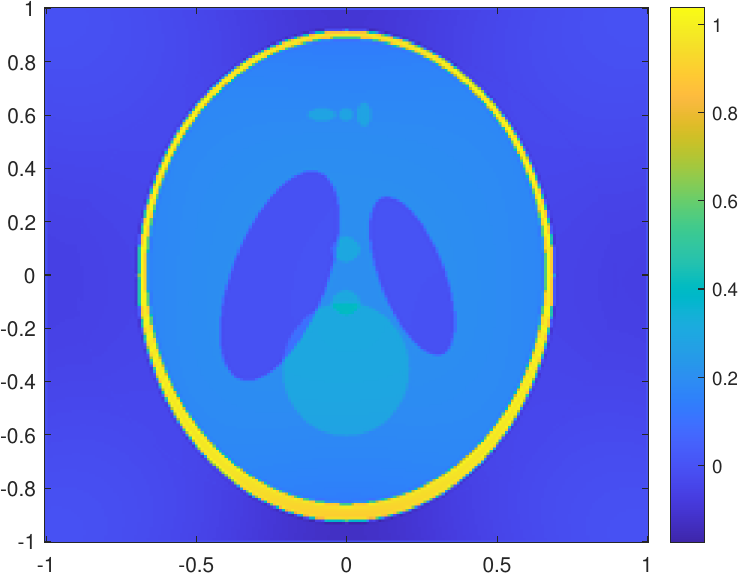}
    \includegraphics[width=0.4\textwidth]{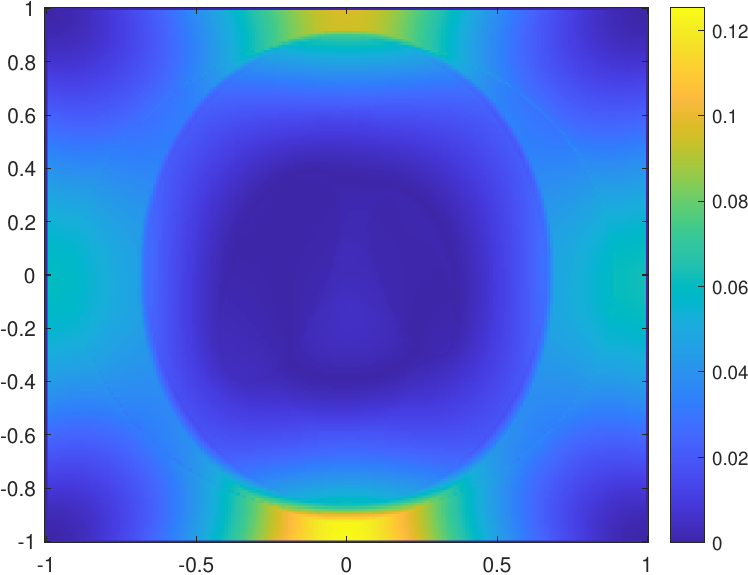}
    \caption{Averaged reconstructed source $\tilde{S}$ and its error under $10\%$ Gaussian random noise.}
    \label{fig:exp2_average}
    
    \includegraphics[width=0.4\textwidth]{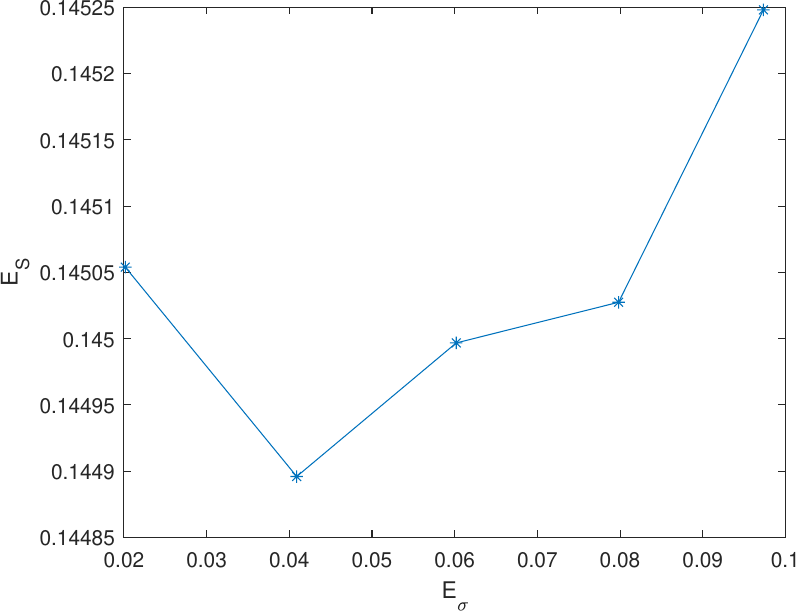}
    \includegraphics[width=0.4\textwidth]{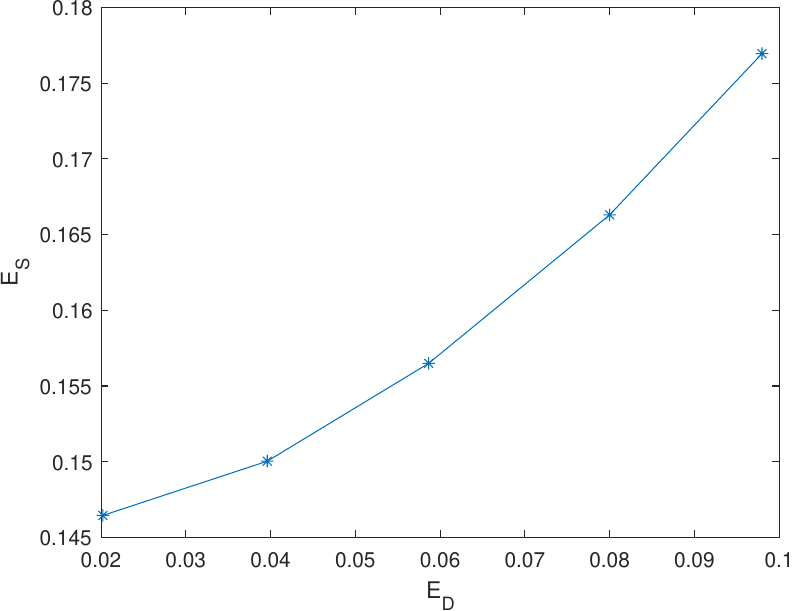}
    \caption{Left: $\mathcal{E}_S$ versus $\mathcal{E}_{\sigma_a}$. Right: $\mathcal{E}_S$ versus $\mathcal{E}_{D}$.}
    \label{fig:exp2_stability}
\end{figure}

\begin{remark} 
In Fig.~\ref{fig:exp2_distribution}, the plot $\|\Delta S\|$ versus $\|\Delta D\|$ has two branches. This is because the plot shows the relation between the norms. As a simple example, 
let $y=(x+1)^2$, $x\in\mathbb{R}$. The same branches appear if we plot $|y|$ versus $|x|$.
\end{remark}

\bibliographystyle{plain} 
\bibliography{main.bib}

\end{document}